\documentclass[11 pt]{amsart}

\usepackage{graphicx}
\usepackage{amsmath}
\usepackage{amscd}
\usepackage{amsfonts}
\usepackage{amssymb}
\usepackage[dvipsnames]{xcolor}
\usepackage{mathtools}
\usepackage{graphicx}
\usepackage{caption}
\usepackage{subcaption}
\usepackage{verbatim}
\usepackage[hypertexnames=false, colorlinks, citecolor=red,linkcolor=blue, urlcolor=red]{hyperref}
\usepackage[margin=1.45in]{geometry}

\numberwithin{equation}{section}

\theoremstyle{definition}
\newtheorem{theorem}[equation]{Theorem}
\newtheorem{lemma}[equation]{Lemma}

\newtheorem{corollary}[equation]{Corollary}

\newtheorem{definition}[equation]{Definition}
\newtheorem{example}[equation]{Example}
\newtheorem{remark}[equation]{Remark}

\newcommand{\be}{\begin{equation}}
\newcommand{\ee}{\end{equation}}
\newcommand{\bes}{\begin{equation*}}
\newcommand{\ees}{\end{equation*}}
\newcommand{\ba}{\begin{aligned}}
\newcommand{\ea}{\end{aligned}}

\newcommand{\cA}{\mathcal{A}}

\newcommand{\cC}{\mathcal{C}}
\newcommand{\cD}{\mathcal{D}}
\newcommand{\cE}{\mathcal{E}}
\newcommand{\cF}{\mathcal{F}}

\newcommand{\cL}{\mathcal{L}}

\newcommand{\cO}{\mathcal{O}}

\newcommand{\cS}{\mathcal{S}}

\newcommand{\cU}{\mathcal{U}}
\newcommand{\cW}{\mathcal{W}}

\newcommand{\FORAL}{\text{ for all }}
\newcommand{\qand}{\quad\text{and}\quad}

\newcommand{\qforal}{\quad\text{for all}\quad}

\newcommand{\bC}{\mathbb{C}}

\newcommand{\bF}{\mathbb{F}}
\newcommand{\bK}{\mathbb{K}}
\newcommand{\bL}{\mathbb{L}}

\newcommand{\bR}{\mathbb{R}}

\newcommand{\bZ}{\mathbb{Z}}

\newcommand{\diag}{\operatorname{diag}}

\newcommand{\esssup}{\operatorname{ess.sup}}
\newcommand{\id}{\operatorname{id}}

\newcommand{\Irr}{\operatorname{Irr}}

\newcommand{\spn}{\operatorname{span}}

\newcommand{\Wmin}[1]{\cW^{\text{min}}_{#1}}
\newcommand{\Wmax}[1]{\cW^{\text{max}}_{#1}}

\newcommand{\ol}{\overline}

\allowdisplaybreaks

\begin{document}

\title[Strongly Peaking Representations and Compressions]{Strongly Peaking Representations and Compressions of Operator Systems}

\author[K.R. Davidson]{Kenneth R. Davidson}
\address{Department of Pure Mathematics\\ University of Waterloo\\Waterloo, ON, N2L 3G1, Canada}
\email{krdavids@uwaterloo.ca}

\author[B. Passer]{Benjamin Passer}
\address{Department of Mathematics\\ United States Naval Academy\\ Annapolis, MD, 21402, United States}
\email{passer@usna.edu}

%%%%%%%%%%%%%%%%%%%%%%%%%%%%
\begin{abstract}
We use Arveson's notion of strongly peaking representation to generalize uniqueness theorems for free spectrahedra and matrix convex sets which admit minimal presentations. A fully compressed separable operator system necessarily generates the $C^*$-envelope and is such that the identity is the direct sum of strongly peaking representations. In particular, a fully compressed presentation of a separable operator system is unique up to unitary equivalence. Under various additional assumptions, minimality conditions are sufficient to determine a separable operator system uniquely. 
\end{abstract}

\thanks{First author supported by NSERC Grant Number 2018-03973.}

\maketitle

%%%%%%%%%%%%%%%%%%%%%%%%%%%%%%%%%%%%%%%
\section{Introduction}
%%%%%%%%%%%%%%%%%%%%%%%%%%%%%%%%%%%%%%%

If $K$ is a compact convex set, then classically one studies the space $A(K)$ of affine functions on $K$ as a unital subspace of $C(K)$. 
The noncommutative analogue is an operator system, which may be defined in the following ``concrete'' way.
An \textit{operator system} $\cS$ is a vector subspace of a unital $C^*$-algebra such that $1 \in \cS$ and $\cS$ is closed under the adjoint operation.

The order structure of an operator system $\cS$ is determined by positive matrices over $\cS$.
The morphisms between operator systems are \textit{UCP} (\textit{unital completely positive}) maps, i.e. linear maps which map the unit to the unit and preserve positivity of matrices of any size. 
Unital order preserving maps are automatically completely contractive, that is, they are contractive on matrices of all sizes over $\cS$. The isomorphisms between operator systems are unital complete order isomorphisms, which are in turn equivalent to unital completely isometric isomorphisms.

An operator system $\cS$ may be embedded into $B(H)$ in various ways, and in particular the $C^*$-algebra that it generates is generally not unique. We write $(\cA,j)$ where $j:\cS \to \cA$ is a unital complete order embedding with $\cA = C^*(j(\cS))$, and we call $(\cA, j)$ a $C^*$-cover of $\cS$. Among the $C^*$-covers of $\cS$, there is always a smallest one, called the $C^*$-envelope $(C^*_e(\cS), \iota)$. This is characterized by the property that whenever $(\cA,j)$ is a $C^*$-cover, there is a (unique) unital $*$-homomorphism $\pi: \cA \to C^*_e(\cS)$ such that $\pi \circ j = \iota$. The analogy in the classical case is that $A(K)$ sits inside $C(\ol{\operatorname{ext}( K)})$, where $\ol{\operatorname{ext}(K)}$, the closure of the extreme points, is the Shilov boundary. This is the smallest closed subset of $K$ on which every affine function attains it maximum modulus. In contrast, the Choquet boundary $\operatorname{ext}(K)$ need not be closed.

The existence of the $C^*$-envelope was conjectured by Arveson in \cite{Arv69, Arv72} and established by Hamana in \cite[Theorem 4.1]{Ham79}. 
However, this construction did not address existence of the noncommutative Choquet boundary. 
If $\cS$ is an operator system which generates a $C^*$-algebra $\cA$, then Arveson also defined that a $*$-representation $\pi$ of $\cA$ has the \textit{unique extension property} (relative to $\cS$) if $\phi = \pi|_{\cS}$ has a unique UCP extension to $\cA$, namely $\pi$. 
He proposed that the analogue of an extreme point should be a boundary representation, an irreducible representation with the unique extension property, and conjectured that the direct sum of all boundary representations would provide a faithful representation of the $C^*$-envelope. 
A new proof of the existence of the $C^*$-envelope by Dritschel and McCullough \cite{DMcC05} shed light on the problem, and in \cite[Theorem 7.1]{Arv08}, Arveson proved there are sufficiently many boundary representations in the separable case. 
The first author and Kennedy established the conjecture in full generality in \cite[Theorem 3.4]{DavKen15}.

In this manuscript, we will consider separable operator systems $\cS$. We associate the matrix state space 
\[ \bK =  \bigcup\limits_{n=1}^\infty K_n,\] 
where
\[ K_n = \{ \phi: \cS \to M_n \,\, | \,\, \phi \text{ is UCP}\}.\]
The matrix state space $\bK$ has the structure of a matrix convex set \cite[Definition 1.1]{WW99}. When $\cS$ is in fact finite-dimensional, there is an operator $d$-tuple $T = (T_1, \ldots, T_d) \in B(H)^d$ such that $\cS = \cS_T$, where
\[
 \cS_T := \text{span}\{I, T_1, \ldots, T_d, T_1^*, \ldots, T_d^*\}. 
\]
The matrix state space of $\cS_T$ is reflected in the \textit{matrix range} of $T$ (see \cite[\S 2.4]{Arv72} and \cite[\S 2.2]{DDSS}). The following definition uses Arveson's extension theorem \cite[Theorem 1.2.3]{Arv69}.

%%%%%%%%%%%%%%%%%%%%%%%%%%%%
\begin{definition}
If $T \in B(H)^d$, then the \textit{matrix range} of $T$ is $\cW(T) := \bigcup\limits_{n=1}^\infty \cW_n(T)$, where
\[
 \begin{aligned}
  \cW_n(T) 	&:= \{ (\phi(T_1), \ldots, \phi(T_d)) \,\, | \,\, \phi: \cS_T \to M_n(\bC) \text{ is UCP}\}  \\
			&\hphantom{:}= \{(\phi(T_1), \ldots, \phi(T_d)) \,\, | \,\, \phi: B(H) \to M_n(\bC) \text{ is UCP}\}. 
 \end{aligned} 
\]
\end{definition}

The matrix state space of $\cS$, similarly the matrix range of $T$ when $\cS = \cS_T$, is sufficient to determine $\cS$ up to complete order isomorphism. 
We are interested in imposing additional conditions on $\cS$, or on $T$, which instead enable us to determine a spatial presentation uniquely up to unitary equivalence. 
The prototype result along these lines is \cite[Theorem 3]{Arv70}, which shows that two irreducible compact operators are unitarily equivalent if and only if they have the same matrix range.
Recall the following definitions from \cite[Definition 6.1]{DDSS} and \cite[Definition 3.20]{Pas19}.

%%%%%%%%%%%%%%%%%%%%%%%%%%%%%
\begin{definition}
A tuple $T \in B(H)^d$ is \textit{minimal} if any nontrivial decomposition $T \cong X \oplus Y$ has $\cW(X) \subsetneq \cW(T)$. Similarly, $T$ is \textit{fully compressed} if whenever $X = P_G T|_G$ is the compression of $T$ to a proper closed subspace $G$ of $H$, it follows that $\cW(X) \subsetneq \cW(T)$.
\end{definition}

Every fully compressed tuple is minimal, but not every minimal tuple is fully compressed, as an arbitrary subspace $G$ need not be reducing. Minimality is sufficient to determine $T$ from $\cW(T)$ if $T$ is assumed to act on a finite-dimensional space. This claim, modulo a polar dual and special treatment of the point $0$, is really a uniqueness claim for free spectrahedra as in \cite[Theorem 1.2]{Zalar} or \cite[Theorem 1.2]{HKM13}. For $T$ which act on infinite-dimensional spaces, minimality is insufficient. See \cite[p.207]{Arv69}, \cite[Example 3.12]{Zalar}, \cite[Example 4.7]{PasSha18}, and \cite[Example 3.14]{Pas19}, noting that the first three examples are irreducible, and the fourth is compact. For compact tuples $T \in K(H)^d$, one may include a nonsingularity assumption to fix the issue \cite[Theorem 6.9]{DDSS} or instead require $T$ to be fully compressed \cite[Corollary 4.3]{PasSha18}. While uniqueness of fully compressed compact tuples may be shown without reference to nonsingularity, a compact tuple is in fact fully compressed if and only if it is both minimal and nonsingular \cite[Theorem 4.4]{PasSha18}. 

The following definition provides a natural extension of minimal and fully compressed tuples to the more general setting of operator systems.

%%%%%%%%%%%%%%%%%%%%%%%%%%%%%
\begin{definition}\label{def:fc_intro_os}
Let $\cS \subset B(H)$ be an operator system. Then $\cS$ is called \textit{fully compressed} if for every proper, closed subspace $G$ of $H$, the compression map $s \mapsto P_G \, s |_{G}$ is not completely isometric on $\cS$. Similarly, $\cS$ is called \textit{minimal} if for every proper, closed,  reducing subspace $G$ of $H$, the compression map $s \mapsto P_G \, s|_G$ is not completely isometric on $\cS$.
\end{definition}

Our goal is to classify fully compressed operator systems up to unitary equivalence, and to further clarify the role of minimality in these uniqueness theorems. In this pursuit, we rely heavily on peaking properties of an operator system.  The concept of a \textit{crucial matrix extreme point} (see Definition \ref{def:crucial_MEP}), a noncommutative generalization of an isolated extreme point of a compact convex set, is used repeatedly in \cite{PasSha18}. Arveson's earlier notion of strongly peaking representation, as in Definition \ref{def:Arv_strongpeak}, plays a similar role but is technically distinct. In particular, it relies upon the concrete presentation of an operator system. We nonetheless use the latter to characterize fully compressed separable operator systems, and consequently fully compressed operator tuples, up to unitary equivalence. Our main results, Theorem \ref{thm:fc_full_equiv_general} and Corollary \ref{cor:fc_uniqueness}, are as follows.

%%%%%%%%%%%%%%%%%%%%%%%%%%%%%
\begin{theorem}
Let $\cS \subset B(H)$ be a separable operator system. Then the following are equivalent.
\begin{enumerate}
\item $\cS$ is fully compressed.

\item The identity representation on $C^*(\cS)$ is unitarily equivalent to $\bigoplus\limits_{\pi \in \Omega} \pi$, where $\Omega$ is some collection of unitarily inequivalent boundary representations such that no $\pi \in \Omega$ annihilates $C^*(\cS) \cap K(H)$.

\item $\cS$ is minimal, $C^*(\cS) = C^*_e(\cS)$, and the identity representation on $C^*(\cS)$ is unitarily equivalent to $\bigoplus\limits_{\pi \in \Xi} \pi$, where $\Xi$ is some collection of irreducible representations $\pi$ such that  $\pi(C^*(\cS)) \cap K(H_\pi) \ne \{0\}$.

\item The identity representation on $C^*(\cS)$ is unitarily equivalent to $\bigoplus\limits_{\pi \in \Lambda} \pi$, where $\Lambda$ consists of the strongly peaking representations without multiplicity.
\end{enumerate}
\end{theorem}

%%%%%%%%%%%%%%%%%%%%%%%%%%%%%
\begin{corollary}
If two separable operator systems are fully compressed and completely order isomorphic, then any complete order isomorphism between them is a unitary equivalence.
\end{corollary}

This corollary should be contrasted with results on approximate unitary equivalence, such as \cite[Theorem 6.11]{DDSS}, \cite[Theorem 1.1]{Zalar}, or \cite[Theorem 6.45]{Kriel}.  We also note that since strongly peaking representations are reliant on the particular $C^*$-cover of $\cS$, one must first use the fact that a fully compressed operator system generates the (unique) $C^*$-envelope in order to derive the corollary. Surprisingly, in Example \ref{ex:Cuntz-Toeplitz} we find that fully compressed operator systems need not generate GCR $C^*$-algebras.

In section \ref{sec:wrapup}, we explore various special cases of the main theorem and discuss its consequences in the language of matrix convexity. Under various GCR assumptions, we may replace fully compressed operator systems with minimal ones, as in Theorem \ref{thm:type I}.

%%%%%%%%%%%%%%%%%%%%%%%%%%%%%
\begin{theorem}
Let $\cS \subset B(H)$ be a separable operator system. Then the following are equivalent.
\begin{enumerate}
\item $\cS$ is fully compressed.
\item $C^*(\cS) = C^*_e(\cS)$ and $\id$ is the direct sum of inequivalent irreducible representations $\pi_i$, where each class $[\pi_i]$ is isolated in $\widehat{C^*(\cS)}$.
\item $\cS$ is minimal, $C^*(\cS) = C^*_e(\cS)$, and $C^*(\cS) \cap K(H)$ is essential.
\item $\cS$ is minimal, $C^*(\cS) = C^*_e(\cS)$, and the maximal GCR ideal $I$ is essential.
\end{enumerate}
\end{theorem}

The GCR condition above generalizes results in \cite[\S 6]{DDSS} and \cite[\S 4]{PasSha18} for compact tuples, and further special cases such as Corollary \ref{cor:countable_nice} are similarly applicable to the compact case. We also find in Theorem \ref{thm:block_diagonal} that if $C^*(\cS)$ is block diagonal, then $\cS$ is minimal if and only if it is fully compressed, without the need to assume $C^*(\cS) = C^*_e(\cS)$. 

In the remainder of section \ref{sec:wrapup}, we describe the relationship between crucial matrix extreme points and (finite-dimensional) strongly peaking representations, as in Theorem \ref{thm:no_prob_envelope}, and use this to derive further special cases. The section closes with a comparison of crucial matrix extreme points with Kriel's \textit{matrix exposed} points \cite[Definition 6.1]{Kriel}, placing our results in the context of matrix convex sets and free spectrahedra.

%%%%%%%%%%%%%%%%%%%%%%%%%%%%%%%%%%%%%%%
\section{Strongly Peaking Representations and Compressions}\label{sec:infinite_dimension}
%%%%%%%%%%%%%%%%%%%%%%%%%%%%%%%%%%%%%%%

In this section, we prove a uniqueness theorem for fully compressed operator systems. Namely, any complete order isomorphism between fully compressed operator systems is in fact a unitary equivalence.

%%%%%%%%%%%%%%%%%%%%%%%%%%%%
\begin{definition}\label{def:fc_enlightened}
Let $\cS \subset B(H)$ be an operator system. Then $\cS$ is called \textit{fully compressed} if for every proper, closed subspace $G$ of $H$, the compression map $s \mapsto P_G \, s |_{G}$ is not completely isometric on $\cS$. Similarly, $\cS$ is called \textit{minimal} if for every proper, closed,  reducing subspace $G$ of $H$, the compression map $s \mapsto P_G \, s|_G$ is not completely isometric on $\cS$.
\end{definition}

To avoid any degeneracy issues, we do not consider the trivial $C^*$-algebra $\{0\}$ to be unital, and hence $\cS = \{0\}$ is not an operator system. For $T \in B(H)^d$ and the corresponding operator system $\cS_T$, Definition \ref{def:fc_enlightened} agrees with the previous notions of minimal and fully compressed tuples.

In \cite[Definition 7.1]{Arv11}, Arveson defines strongly peaking representations of a separable operator system to fill the role of isolated extreme points in the noncommutative setting. Here, we let $\Irr(\cA)$ consist of irreducible representations of $\cA$, and we use $\cong$ to denote unitary equivalence. We also denote the ampliations of a representation $\pi$ with the same symbol. That is, $\pi$ may be evaluated at any matrix over $\cA$.

%%%%%%%%%%%%%%%%%%%%%%%%%%%%
\begin{definition}\label{def:Arv_strongpeak}
Let $\cS$ be a separable operator system. A representation $\pi: C^*(\cS) \to B(H_\pi)$ is called \textit{strongly peaking} if $\pi$ is irreducible and there exists a matrix $S \in M_n(\cS)$ such that
\[
  \| \pi(S)  \| > \sup \big\{  \| \sigma(S) \| : \sigma \in \operatorname{Irr}(C^*(\cS)),\  \sigma \not\cong \pi \big\}.
\]
\end{definition}

%%%%%%%%%%%%%%%%%%%%%%%%%%%%
\begin{remark}\label{remark:sp_imp_acc}
Strongly peaking representations are by definition included in any collection of irreducible representations of $C^*(\cS)$ which completely norms the separable operator system $\cS$. 
Following \cite[Remark 2.5 and Lemma 3.3]{DavKen15}, $\cS$ may be completely normed by boundary representations which are obtained by a finite or countably infinite dilation $\phi_1 \prec \phi_2 \prec \ldots$ of pure matrix states, i.e. pure UCP maps from $\cS$ into matrix algebras, such that the dilation culminates in a pure and maximal UCP map. 
To be consistent with \cite[\S 6]{Kriel}, we call any boundary representation $\pi$ obtained by this dilation procedure \textit{accessible}, whether the dimension of $\pi$ is finite or infinite. 
In particular, every strongly peaking representation is an accessible boundary representation.
\end{remark}

For $\cS_T$, Arveson's definition of strongly peaking is equivalent to the condition
\[
 \cW(\pi(T)) \not\subseteq \cW \Big( \bigoplus_{\substack{\sigma\in\Irr(C^*(\cS))\\ \sigma\not\cong \pi}} \sigma(T) \Big).
\]
Here $\pi$ is an irreducible representation, and the direct sum ranges over irreducible representations which are not unitarily equivalent to $\pi$.

Strongly peaking representations are irrevocably tied to the $C^*$-cover of $\cS$. Every strongly peaking representation $\pi$ of $\cS \subset C^*(\cS)$ factors through the $C^*$-envelope as $\pi = \widetilde{\pi} \, j$, since $\pi$ is boundary, and in this case $\widetilde{\pi}$ is also strongly peaking for $j(\cS) \subset C^*_e(\cS)$. However, it is possible for $\widetilde{\pi}$ to be strongly peaking even though $\pi$ is not. Thus, the $C^*$-envelope has the largest possible family of strongly peaking representations. As seen in \cite[Theorem 7.2]{Arv11}, strongly peaking representations are also closely tied to compact operators, and we will expand upon this result.

%%%%%%%%%%%%%%%%%%%%%%%%%%%%
\begin{lemma}\label{lem:peaking and compacts}
Let $\cS \subset B(H)$ be a separable operator system, let $\pi$ be a strongly peaking representation for $\cS$, and let $J := \bigcap\limits_{\substack{\sigma\in\Irr(C^*(\cS))\\ \sigma\not\cong \pi}} \ker \sigma$. Then the following hold.
\begin{itemize}
\item $\pi(C^*(\cS))$ contains $K(H_\pi)$. In particular, $\pi(J) = K(H_\pi)$ and $\pi|_J$ is injective.
\item $\pi$ is a subrepresentation of the identity, specifically $\id \cong \pi^{(n)} \oplus \tau$ where $n \in \bZ^+ \cup \{\aleph_0\}$ and $\tau$ is some (possibly vacuous) representation satisfying $\tau|_J = 0$.
\item If $C^*(\cS)$ is multiplicity free, then $J$ is a summand of $C^*(\cS) \cap K(H)$, hence $\pi$ is induced by its restriction to $C^*(\cS) \cap K(H)$.
\end{itemize}
\end{lemma}

\begin{proof}
Let $S = [s_{ij}] \in M_n(\cS)$ be chosen so that
\[
 \| \pi(S) \| = 1 > r := \sup_{\substack{\sigma\in\Irr(C^*(\cS))\\ \sigma\not\cong \pi}} \| \sigma(S)\|.
\]
Define $J$ as above, and note $\pi$ does not vanish on $J$. Indeed, if $f(x) = 0$ on $[0,r^2]$ and $f(x) = (x-r^2)/(1-r^2)$ on $[r^2,1]$, then $A := f(S^*S)$ belongs to $M_n(J)$, but $\|\pi(A)\|=1$.

Since $\pi$ is irreducible, so is $\pi|_J$. Moreover, every irreducible representation of $J$ lifts to a (unique) irreducible representation of $C^*(\cS)$ by \cite[Lemma I.9.14]{DavBook}. Except for the unitary equivalence class of $\pi$, every other irreducible representation of $C^*(\cS)$ vanishes on $J$, so $\pi|_J$ is the unique irreducible representation of $J$ up to unitary equivalence. By Rosenberg's solution of the Naimark problem \cite[Theorem 4]{Ros53}, $J$ is isomorphic to the compact operators on some separable or finite-dimensional Hilbert space. It follows that $\pi|_J$ is equivalent to the identity representation on $J$, and hence $\pi(J) = K(H_\pi) \subseteq \pi(C^*(\cS))$. In particular, $\pi|_J$ is injective.

By \cite[Lemma 2.11.1]{DixmierBook}, we may decompose $\id \cong \mu \oplus \tau$ where $\mu|_J$ is nondegenerate and $\tau|_J = 0$. Since $J$ is isomorphic to the compact operators $K(H_\pi)$, $\mu|_J$ is equivalent to a multiple of the identity representation, that is, a multiple of $\pi|_J$. It follows that $\mu$ is a multiple of $\pi$ with some positive multiplicity. If $C^*(\cS)$ is multiplicity free, then this multiplicity is exactly 1, and thus $\id \cong \pi \oplus \tau$ where $\tau|_J = 0$. Hence $\id(J) = K(H_\pi) \oplus 0 \subset C^*(\cS) \cap K(H)$, and $\pi$ is induced by its restriction to $C^*(\cS) \cap K(H)$.
\end{proof}

The reader is cautioned that in Lemma \ref{lem:peaking and compacts}, the strongly peaking representations are particularly sensitive to the concrete presentation. We will need the following consequence of the lemma.

%%%%%%%%%%%%%%%%%%%%%%%%%%%%
\begin{corollary}\label{cor:peaking_projection}
Let $\cS \subset B(H)$ be a separable operator system, and let $\pi: C^*(\cS) \to B(H_\pi)$ be a strongly peaking representation for $\cS$.
Suppose that $v \in H_\pi$ is an arbitrary unit vector. 
Then there is a projection $P \in C^*(\cS)$ such that $\pi(P) = vv^*$ and $\sigma(P) = 0$ for any $\sigma \in \Irr(C^*(\cS))$ with $\sigma \not\cong \pi$. If, in addition, $C^*(\cS)$ is multiplicity free, then $P$ has rank one.
\end{corollary}

\begin{proof}
As in Lemma \ref{lem:peaking and compacts}, let $J$ be the intersection of $\ker \sigma$ over all irreducible representations $\sigma \not\cong \pi$, so that $\pi(J) = K(H_\pi)$. We may also decompose $\id \cong \pi^{(n)} \oplus \tau$ where $n \in \bZ^+ \cup \{\aleph_0\}$ and $\tau|_J = 0$.  
Since $vv^*$ has rank one, there is some $P \in J$ with $\pi(P) = vv^* \in K(H_\pi)$, and certainly $\sigma(P) = 0$ for irreducible $\sigma \not\cong \pi$ by the definition of $J$. 
Finally, $P = \id(P)$ is the direct sum of $0$ and $n$ copies of $\pi(P) = vv^*$, so $P$ is a projection. 
If $C^*(\cS)$ is multiplicity free, then $n = 1$, whence $P$ has rank one.
\end{proof}

Similar techniques also provide a partial converse to Lemma \ref{lem:peaking and compacts}. We note in particular that the compression condition below applies whenever $\cS$ is minimal.

%%%%%%%%%%%%%%%%%%%%%
\begin{corollary} \label{cor:envelope and minimal}
Let $\cS \subset B(H)$ be a separable operator system, and let $\pi$ be an irreducible subrepresentation of the identity representation of $C^*(\cS)$. Assume that $\pi(C^*(\cS)) \cap K(H_\pi) \ne \{0\}$ and that the compression of $\cS$ to the reducing subspace $H_\pi^\perp$ is not completely isometric.  Then $\pi$ is induced by its restriction to $C^*(\cS) \cap K(H)$. If, in addition, $\pi$ factors through the $C^*$-envelope, then $\pi$ is strongly peaking for $\cS$.
\end{corollary}

\begin{proof}
We may write $\id \cong \pi \oplus \tau$ with respect to $H = H_\pi \oplus H_\tau$, where $\tau$ might be vacuous. Since the compression to $H_\pi^\perp$ is not completely isometric on $\cS$, there is a matrix $S=[s_{ij}] \in M_n(\cS)$ with
$\| \pi(S)\| > \|\tau(S) \| $.
Let $J = \ker\tau$. Similar reasoning as in the proof of Lemma~\ref{lem:peaking and compacts} shows that $\pi$ is nonzero on $J$. Since we have assumed $\pi(C^*(\cS)) \cap K(H_\pi) \not= \{0\}$, i.e. $\pi(C^*(\cS)) \supseteq K(H_\pi)$ since $\pi$ is irreducible, $K(H_\pi)$ is the smallest nonzero ideal of $\pi(C^*(\cS))$. It follows that $\pi(J) \supseteq K(H_\pi)$. 

Set $J_1 := J \cap \pi^{-1}(K(H_\pi))$. Then $\pi(J_1) = K(H_\pi)$, so we also have that
\[
 J_1 = \id(J_1) = \pi(J_1) \oplus \{0\} = K(H_\pi)\oplus\{0\} \subseteq C^*(\cS) \cap K(H) .
\]
Therefore, $\pi$ is induced by its restriction to $C^*(\cS) \cap K(H)$.

Next, define $J_0 := J_1 \cap \ker \pi$. Then $\tau(J_0) = \{0\} = \pi(J_0)$, so $J_0 = \id(J_0) = \{0\}$, and hence $J_1 \cong J_1/J_0 \cong K(H_\pi)$.
Thus, $J_1$ has a unique irreducible representation, namely $\pi|_{J_1}$, up to unitary equivalence. 
It follows from \cite[Lemma I.9.14]{DavBook} that $\pi$ is the unique irreducible representation of $C^*(\cS)$ that does not vanish on $J_1$. 
That is, every other irreducible representation of $C^*(\cS)$ annihilates $J_1$.

If $\pi$ factors through the $C^*$-envelope, then the Shilov ideal (the kernel of the natural quotient map $C^*(\cS) \to C^*_e(\cS)$) is contained in $\ker \pi$. 
The Shilov ideal is the unique largest ideal such that the quotient is completely isometric on $\cS$. 
Since $J_1 \not\subseteq \ker \pi$, the quotient by $J_1$ is not completely isometric on $\cS$. 
Let $S = [s_{ij}]\in M_n(\cS)$ be chosen so that $1 = \|S\| > \| S + M_n(J_1) \|$.
Then since any irreducible $\sigma\not\cong \pi$ vanishes on $J_1$,
\[
 1 = \|S\|  > r := \| S + M_n(J_1) \| \ge  \sup_{\substack{\sigma\in\Irr(C^*(\cS))\\ \sigma\not\cong \pi}} \| \sigma(S) \| . 
\]
Finally, $J_1 \subseteq J = \ker \tau$ implies $\| \tau(S) \| \leq \|S + M_n(J_1)\| = r$, so $\id \cong \pi \oplus \tau$ gives
\[
 \|S\| = \max \{ \|\pi(S)\|, \|\tau(S)\| \} \le \max \{ \|\pi(S)\|, r \},
\]
and hence $\| \pi(S) \| = 1 > r$. Therefore, $\pi$ is strongly peaking.
\end{proof}

%%%%%%%%%%%%%%%%%%%%%%%%%%%%
\begin{example} \label{ex:tight}
The $C^*$-envelope assumption is required in the previous corollary to conclude that $\pi$ is strongly peaking. 
This follows from \cite[Example 3.14]{Pas19}, given as Example~\ref{ex:nonpeaking_scum} below, in which $\cS$ is minimal and $\id$ decomposes
as the direct sum of three irreducible representations. Each subrepresentation contains the compact operators, but one of the
representations fails to factor through the $C^*$-envelope. That representation is not strongly peaking.

The assumption that the range of $\pi$ contain the compacts is also necessary.
Let $\cS$ be the operator system spanned by the standard generators of the Cuntz algebra $\cO_2$.
Since $\cO_2$ is simple, if $\pi$ is an arbitrary irreducible representation of $\cO_2$, then $\pi$ is automatically faithful, hence it is completely isometric on $\cS$. 
Moreover, $\pi(C^*(\cS)) \cap K(H_\pi) = \{0\}$. 
This system is hyperrigid \cite[Corollary 3.4]{Arv11}, and thus every irreducible representation is a boundary representation. 
However, no representation is strongly peaking, as there are many distinct irreducible representations up to unitary equivalence, all of which are faithful. 
Thus, our results are consistent with \cite[Example 3.12]{Zalar}.
\end{example}

The next two lemmas identify the role of compact operators with regard to proper compressions. 

%%%%%%%%%%%%%%%%%%%%%%%%%%%%
\begin{lemma}\label{lem:detection_dichotomy}
Let $\cS \subset B(H)$ be a separable operator system, and suppose $\pi: C^*(\cS) \to B(H_\pi)$ is an irreducible representation. Then either
\begin{enumerate}
\item\label{item:detection_fun_compacts} $\pi$ does not annihilate $C^*(\cS) \cap K(H)$, in which case $\pi$ is unitarily equivalent to a subrepresentation of the identity representation, or
\item\label{item:detection_fun_no_compacts} $\pi$ annihilates $C^*(\cS) \cap K(H)$, in which case for any subspace $G$ of $H$ of finite codimension, there is a UCP map $\phi : P_G \cS |_G \to B(H_\pi)$ given by $\phi( P_G \, s|_G) = \pi(s)$ for $s \in \cS$. 
\end{enumerate}
\end{lemma}

\begin{proof}
Let $J = C^*(\cS) \cap K(H)$. 
By \cite[Lemma 2.11.1]{DixmierBook}, we may decompose $\pi \cong \pi_a \oplus \pi_s$, where $\pi_a$ acts on $\ol{ \pi(J)H}$ and $\pi_s$ acts on $(\pi(J)H)^\perp$. Since $\pi$ is irreducible, either $\pi = \pi_a$ or $\pi = \pi_s$.

Whenever $\pi = \pi_a$, $\pi|_J$ is an irreducible representation of a $C^*$-algebra of compact operators, so it is unitarily equivalent to a subrepresentation of the identity representation on $J$ by \cite[Theorem I.10.7]{DavBook}. 
In particular, there is an isometry $V: H_\pi \to H$ such that $\pi(b) = V^* b V$ for all $b \in J$. 
The subspace $VH$ reduces $J$ and hence reduces $C^*(\cS)$. 
Moreover, the extension of $\pi|_J$ from $J$ to $C^*(\cS)$ is unique by \cite[Lemma I.9.14]{DavBook}, and thus 
\[ \pi(a) = V^* a V \qforal a \in C^*(\cS). \]
Therefore, $\pi$ is unitarily equivalent to the compression of the identity representation to $VH$.
So $\pi$ is a subrepresentation of the identity.

On the other hand, suppose that $\pi = \pi_s$. Then $\pi$ factors as $\pi' q$ where $q$ is the quotient map onto $C^*(\cS)/J$.
If $G$ has finite codimension, then the compression onto $G$ annihilates only finite rank operators, so the map $q': P_G \cS |_G \to C^*(\cS)/J$ given by $q'( P_G \, s|_G) = s + J$ is UCP.
Hence the map $\phi = \pi' q'$ is UCP.
\end{proof}

For boundary representations, we can extend the previous lemma a bit more.  

%%%%%%%%%%%%%%%%%%%%%%%%%%%%
\begin{lemma}\label{lem:more_complete_dichotomy}
Suppose $\cS$ is a separable operator system and $\pi: C^*(\cS) \to B(H_\pi)$ is a boundary representation for $\cS$. Then the following are equivalent.
\begin{enumerate}
\item\label{condition:no_cpt_pi} $\pi(C^*(\cS)) \cap K(H_\pi) = \{0\}$.
\item\label{condition:all_compressions_cof} $H_\pi$ is infinite-dimensional, and for any cofinite-dimensional subspace $G$ of $H_\pi$, the compression map is completely isometric on $\pi(\cS)$.
\item\label{condition:one_proper_compression} There exists a proper closed subspace $G$ of $H_\pi$ such that the compression map is completely isometric on $\pi(\cS)$. 
\end{enumerate}
\end{lemma}

\begin{proof}
$\textbf{(\ref{condition:no_cpt_pi})} \Rightarrow \textbf{(\ref{condition:all_compressions_cof})}$: 
If $\pi(C^*(\cS)) \cap K(H_\pi) = \{0\}$, then  $H_\pi$ must be infinite-dimensional, as otherwise the unit is compact. 
Further, the identity representation on $\pi(C^*(\cS))$ is irreducible and annihilates $\pi(C^*(\cS)) \cap K(H_\pi)$, 
so Lemma \ref{lem:detection_dichotomy} shows any cofinite-dimensional compression map is completely isometric on $\pi(\cS)$.

$\textbf{(\ref{condition:all_compressions_cof})} \Rightarrow \textbf{(\ref{condition:one_proper_compression})}$: This is trivial.

$\textbf{(\ref{condition:one_proper_compression})} \Rightarrow \textbf{(\ref{condition:no_cpt_pi})}$: Suppose $G$ is a proper closed subspace of $H_\pi$ such that the compression map $q$ restricts to a complete isometry on $\pi(\cS)$. This implies there is a UCP map $\phi: P_G \pi(s)|_G \mapsto \pi(s)$, $s \in \cS$. Fix a unit vector $w \in G^\perp$ and let $p$ be the rank one projection onto $\bC w$, noting $q(p) = 0$. 

If $\pi(C^*(\cS))\cap K(H_\pi) \not= \{0\}$, then irreducibility of $\pi$ gives that $\pi(C^*(\cS))$ contains $K(H_\pi)$. Hence, there is some $a \in C^*(\cS)$ with $\pi(a) = p$ and consequently $q(\pi(a)) = 0$. On the other hand, by Arveson's extension theorem, there is a UCP extension $\widetilde{\phi}: B(G) \to B(H_\pi)$ of $\phi$, so $\widetilde{\phi} \, q \, \pi$ is a UCP extension of $\pi|_\cS$ to $C^*(\cS)$. Since $\pi(a) = p \not= 0$ but $\widetilde{\phi}\,q \,\pi(a) = \widetilde{\phi}(0) = 0$, this contradicts the fact that $\pi$ has the unique extension property.
\end{proof}

%%%%%%%%%%%%%%%%%%%%%%%%%%%%
\begin{example}
In this lemma, it is not sufficient to assume that $\pi$ is irreducible. 
Let $S$ be the unilateral shift on $\ell^2$, and set $\cS = \spn\{ I, S, S^* \}$.
Then $C^*(\cS)$ is the Toeplitz algebra, which contains all compact operators.
That is, the identity map is an irreducible representation whose range contains the compacts.
Nevertheless, the compression to any subspace of finite codimension is completely isometric on $\cS$.
Indeed, the compression to the complement of the first $n$ standard basis vectors is clearly completely isometric, and as $n$ increases, these spaces will almost subsume any finite-dimensional subspace.
\end{example}

We now finally have enough preparation to reach the main theorem.

%%%%%%%%%%%%%%%%%%%%%%%%%%%%
\begin{theorem}\label{thm:fc_full_equiv_general}
Let $\cS \subset B(H)$ be a separable operator system. Then the following are equivalent.
\begin{enumerate}
\item\label{item:fc_stuff} $\cS$ is fully compressed.

\item\label{item:classification_stuff} The identity representation on $C^*(\cS)$ is unitarily equivalent to $\bigoplus\limits_{\pi \in \Omega} \pi$, where $\Omega$ is some collection of unitarily inequivalent boundary representations such that no $\pi \in \Omega$ annihilates $C^*(\cS) \cap K(H)$.

\item\label{item:more_boundary_nonsense} $\cS$ is minimal, $C^*(\cS) = C^*_e(\cS)$, and the identity representation on $C^*(\cS)$ is unitarily equivalent to $\bigoplus\limits_{\pi \in \Xi} \pi$, where $\Xi$ is some collection of irreducible representations $\pi$ such that  $\pi(C^*(\cS)) \cap K(H_\pi) \ne \{0\}$.

\item\label{item:individual_classification} The identity representation on $C^*(\cS)$ is unitarily equivalent to $\bigoplus\limits_{\pi \in \Lambda} \pi$, where $\Lambda$ consists of the strongly peaking representations without multiplicity.
\end{enumerate}
\end{theorem}

\begin{proof} 
$\textbf{(\ref{item:fc_stuff})} \!\Rightarrow\! \textbf{(\ref{item:classification_stuff})}$: Let $\Omega$ consist of the unitary equivalence classes of all boundary representations $\pi$ of $\cS$ such that $\pi$ is a compression of $\text{id}_{C^*(\cS)}$. 
Then the associated subspaces $H_\pi \subseteq H$ are reducing and mutually orthogonal by \cite[Proposition 5.2.1]{DixmierBook}.

Suppose for contradiction that the $H_\pi$ do not span $H$. 
Fix a unit vector $v$ orthogonal to $\sum_{\pi\in\Omega} H_{\pi}$, and let $G = (\bC v)^\perp$. 
The compression map to $G$ cannot be completely isometric because $\cS$ is fully compressed. Since the direct sum of all boundary representations completely norms $\cS$,
there must be some boundary representation $\rho$ and some $S \in M_n(\cS)$ such that 
\[
 \|\rho(S)\| > \| P_G S|_G \| .
\]
Consequently, $\rho$ fails Lemma \ref{lem:detection_dichotomy} condition (\ref{item:detection_fun_no_compacts}), so condition (\ref{item:detection_fun_compacts}) of the same lemma shows $\rho$ does not annihilate $C^*(\cS) \cap K(H)$, and hence $\rho$ is unitarily equivalent to a subrepresentation of the identity. 
This implies that $\rho \in \Omega$, so $H_\rho \subseteq G$, which contradicts the above inequality. 
It follows that
\[
 \id_{C^*(\cS)} \cong \bigoplus_{\pi \in \Omega} \pi. 
\]

We may similarly apply Lemma \ref{lem:detection_dichotomy} to each $\pi \in \Omega$. If $\pi$ annihilates $C^*(\cS) \cap K(H)$, then given $v \in H_\pi$ and $G = (\bC v)^\perp \subset H$, there is a UCP map $P_G s|_G \mapsto \pi(s)$. For $\sigma \in \Omega$ with $\sigma \not\cong \pi$, $\sigma(s)$ is a compression of $P_G s|_G$, so the map $P_G s|_G \mapsto s$ is UCP. This contradicts the fact that $\cS$ is fully compressed.

$\textbf{(\ref{item:classification_stuff})} \!\Rightarrow\! \textbf{(\ref{item:more_boundary_nonsense})}$: 
Assume $\text{id}_{C^*(\cS)}  \cong \bigoplus\limits_{\pi \in \Omega} \pi$ for unitarily inequivalent boundary representations $\pi \in \Omega$ which do not annihilate $C^*(\cS) \cap K(H)$. This implies that $C^*(\cS) = C^*_e(\cS)$ and $C^*(\cS)$ is multiplicity free, so the (minimal) reducing subspaces of $C^*(\cS)$ are the $H_\pi$. Since there exist boundary representations that live on $C^*(\cS) \cap K(H)$, \cite[Theorem 7.2]{Arv11} implies that the quotient by $C^*(\cS) \cap K(H)$ is not completely isometric on $\cS$, and each $\pi \in \Omega$ is strongly peaking for $\cS$. It follows that no summand may be removed without decreasing the norm of some matrix over $\cS$, so $\cS$ is minimal. Finally, every $\pi \in \Omega$ has $\pi(C^*(\cS)) \cap K(H_\pi) \not= \{0\}$ by Lemma \ref{lem:peaking and compacts}.

$\textbf{(\ref{item:more_boundary_nonsense})} \!\Rightarrow\! \textbf{(\ref{item:individual_classification})}$:
Since $\cS$ is minimal, $C^*(\cS)$ must be multiplicity free.
By Lemma~\ref{lem:peaking and compacts}, each strongly peaking representation belongs to $\Xi$ exactly once up to unitary equivalence.
Suppose that $\pi \in \Xi$.
Since $\cS$ is minimal, $C^*(\cS)=C^*_e(\cS)$, and the range of $\pi$ contains the compacts, 
Corollary~\ref{cor:envelope and minimal} shows that $\pi$ must be strongly peaking.
Thus $\Xi = \Lambda$.

$\textbf{(\ref{item:individual_classification})} \!\Rightarrow\! \textbf{(\ref{item:fc_stuff})}$: Assume that $\id$ is the direct sum of inequivalent strongly peaking representations $\pi \in \Lambda$, but $G$ is a proper closed subspace of $H$ such that the compression to $G$ is a complete isometry on $\cS$.
We may suppose $G$ has codimension one and write it as $G = (\bC w)^\perp$ for some unit vector $w = (w_\pi)_{\pi \in \Lambda}$.

Select some $\pi \in \Lambda$ with $w_\pi \ne 0$, and let $v = \|w_\pi\|^{-1} w_\pi$.
By Corollary~\ref{cor:peaking_projection}, there is some $Q \in C^*(\cS)$ such that $\pi(Q) = vv^*$ and $\sigma(Q) = 0$ for $\sigma \in \Lambda \setminus \{\pi\}$. 
That is, 
\[
 Q = \id(Q) = \pi(Q) \oplus \sum_{\sigma \in \Lambda\setminus\{\pi\}} \sigma(Q) = vv^*
\]
is the projection of $H$ onto $\bC v$. By design, $v \not\in G$, and hence $\|P_G  Q|_G\| < 1 = \|Q\|$.

Since $s \mapsto P_G s|_G$ is completely isometric, the map $\phi: P_G s|_G \mapsto s$, $s \in \cS$ is UCP. Extend $\phi$ to a UCP map $\tilde\phi: B(G) \to B(H)$ by Arveson's extension theorem. If $q_G$ is the compression of $C^*(\cS)$ to $G$, then $\tilde\phi \, q_G$ is a UCP extension of $\id|_{\cS}$ to $C^*(\cS)$. This extension differs from $\id$ at $Q$, since $\| P_G Q|_G \| < \|Q\|$ implies that no UCP map, in particular $\widetilde{\phi}$, can send $P_G Q|_G$ to $Q$. Therefore, $\id$ fails to have the unique extension property. However, every strongly peaking representation is a boundary representation (see Remark \ref{remark:sp_imp_acc}) and hence has the unique extension property. Since $\id$ is the direct sum of such representations, it must have the unique extension property by \cite[Proposition 4.4]{Arv11}, which is a contradiction. Therefore, $\cS$ is fully compressed.
\end{proof}

We are now able to prove a uniqueness result for fully compressed separable operator systems.

%%%%%%%%%%%%%%%%%%%%%%%%%%%%
\begin{corollary}\label{cor:fc_uniqueness}
If two separable operator systems are fully compressed and completely order isomorphic, then any complete order isomorphism between them is a unitary equivalence.
\end{corollary}

\begin{proof}
If $\cS$ and $\widetilde{\cS}$ are completely order isomorphic, then they determine isomorphic $C^*$-envelopes $(C^*_e(\cS), \iota)$ and
$(C^*_e(\widetilde{\cS}), \tilde\iota)$. Specifically, there is a $*$-isomorphism $\tau: C^*_e(\cS) \to C^*_e(\widetilde{\cS})$ such that $\tau$ extends the given complete order isomorphism and $\tau\iota = \tilde\iota$. Theorem \ref{thm:fc_full_equiv_general} shows both operator systems live on their $C^*$-envelopes, hence strongly peaking representations of both operator systems are equivalent, in that a strongly peaking representation $\pi$ for $\cS$ corresponds to $\tilde\pi = \pi\tau^{-1}$. Unitary equivalence then follows from item (\ref{item:individual_classification}) of Theorem \ref{thm:fc_full_equiv_general}.
\end{proof}

Complete order isomorphism of operator systems is determined by the matrix state space, or the matrix range when $\cS = \cS_T$. Thus, we reach a positive answer to \cite[Question 4.8]{PasSha18}.

%%%%%%%%%%%%%%%%%%%%%%%%%%%%
\begin{corollary}\label{cor:fc_uniqueness_tuples}
Let $S = (S_1,\dots,S_d)$ and $T=(T_1,\dots,T_d)$ be two fully compressed $d$-tuples such that $\cW(S)=\cW(T)$.
Then $S$ and $T$ are unitarily equivalent.
\end{corollary}
\begin{proof}
Fully compressed $d$-tuples only exist on finite-dimensional or separable Hilbert spaces, and if $\cW(S) = \cW(T)$, the separable operator systems $\cS_S$ and $\cS_T$ are completely order isomorphic via the map which sends $S$ to $T$ \cite[Theorem 5.1]{DDSS}. Thus, we may apply Corollary \ref{cor:fc_uniqueness}.
\end{proof}

The fact that the compact operators play a key role in our theorem suggests that perhaps $C^*_e(\cS)$ must be GCR when $\cS$ is fully compressed.
This is not the case, as the following example shows.

%%%%%%%%%%%%%%%%%%%%%%%%%%%%
\begin{example} \label{ex:Cuntz-Toeplitz}
Let $L_1$ and $L_2$ be the left creation operators on the Fock space $F_2 = \ell^2(\bF_2^+)$.
That is, $\bF_2^+$ is the free semigroup on two generators, namely all words in the alphabet $\{1,2\}$ with the empty word $\emptyset$ as its unit.
Then $F_2$ has an orthonormal basis $\{ \xi_w : w \in \bF_2^+ \}$, and the creation operators are defined by $L_i \xi_w = \xi_{iw}$.
It is well known that $C^*(L_1,L_2)$ is the Cuntz-Toeplitz algebra $\cE_2$, which is a nontrivial extension of the compact operators by $\cO_2$.

Let $P_0 := \xi_\emptyset \xi_\emptyset^* = I - L_1L_1^* - L_2L_2^*$ and define $\cS := \spn\{ I, P_0, L_1, L_2, L_1^*, L_2^* \}$. 
Now $\cE_2$ contains the compact operators as its only proper ideal, and $\cE_2 /K(F_2) \cong \cO_2$ is NGCR, so Lemma \ref{lem:peaking and compacts} shows that no representation which factors through $\cO_2$ will be strongly peaking.
There is a unique representation which does not annihilate the compacts, namely $\id$.
Moreover, it is evident that $ \|P_0\| = 1$ and $\sigma(P_0)=0$ if $\sigma\not\cong\id$ is irreducible.
Thus $\id$ is strongly peaking and $\cS$ is fully compressed.
\end{example}

%%%%%%%%%%%%%%%%%%%%%%%%%%%%
\begin{example} \label{ex:Cuntz}
On the other hand, the Cuntz system of Example~\ref{ex:tight} has no fully compressed presentation.
This is because $C^*_e(\cS) = \cO_2$ has no strongly peaking representations,
or even any representations that contain any compact operators.
\end{example}

%%%%%%%%%%%%%%%%%%%%%%%%%%%%%%%%%%%%%%%
\section{Concrete Operator Systems}\label{sec:wrapup}
%%%%%%%%%%%%%%%%%%%%%%%%%%%%%%%%%%%%%%%

In this section, we clarify the relationship between fully compressed operator systems and GCR $C^*$-algebras, 
and we further discuss our results in the language of matrix convexity. 
While fully compressed operator systems need not live in GCR $C^*$-al\-ge\-bras, examination of the largest GCR ideal allows for sharper results. 
As described in \cite[\S 1.5]{ArvBook}, the image of an irreducible representation $\pi$ of a GCR $C^*$-algebra always contains the compact operators, 
and moreover $\pi$ is determined up to unitary equivalence by its kernel.

Let $\widehat{\cA}$ consist of unitary equivalence classes of irreducible representations of $\cA$.
There is a natural map from $\widehat{\cA}$ onto the space $\operatorname{Prim}(\cA)$ of primitive ideals 
given by $[\pi] \mapsto \ker \pi$. 
The Jacobson hull-kernel topology on $\operatorname{Prim}(\cA)$ thus induces a topology on $\widehat{\cA}$: 
if $X \subseteq \widehat{\cA}$, then 
\[ 
 [\pi] \in \ol{X} \quad\iff\quad  \bigcap_{[\rho] \in X} \ker \rho  \subseteq \ker \pi.
\]

When $\cA$ is not GCR, this topology can be rather uninteresting. 
However, when $\cA$ is GCR, the kernel map from $\cA$ to $\operatorname{Prim}(\cA)$ is a bijection \cite[Proposition~1.5.4]{ArvBook}, 
and by definition it is a homeomorphism.
In the GCR case, there is a Borel cross section from $\widehat{\cA}$ into $\Irr(\cA)$ which is unique up to the action of a Borel unitary valued function. 
This leads to a theory of direct integrals over the spectrum, and the fact that all (separable) representations are unitarily equivalent to a direct integral. 
The interested reader is referred to Arveson's treatment in \cite[\S 4.3]{ArvBook}.

%%%%%%%%%%%%%%%%%%%%%%%%%%%%
\begin{lemma} \label{lem:direct integral}
If $\sigma$ is a direct integral representation of  a GCR $C^*$-algebra $\cA$ with respect to a nonatomic probability measure $\mu$,
then there is a proper reducing subspace of $\sigma(\cA)$ such that the compression of $\sigma(\cA)$ to this subspace is (completely) isometric.
\end{lemma}

\begin{proof}
We are given that $\sigma = \int^\oplus \zeta \,d\mu(\zeta)$.
Let $(a_i)_{i \in \bZ^+}$ be a countable dense sequence in $\cA_{sa}$, with each element repeated infinitely often.
For each $i$, the function $\zeta \to \| \zeta(a_i) \|$ is Borel and $\|\sigma(a_i)\| = \esssup \|\zeta(a_i)\|$ by \cite[p.89]{ArvBook}.
Thus 
\[
 A_i := \{ \zeta \in \widehat{\cA} : \|\zeta(a_i)\| > \|\sigma(a_i) \| - 2^{-i} \}
\] 
is a Borel set with $\mu(A_i)>0$.

Since $\mu$ has no atoms, we may select a Borel subset $B_i \subset A_i$ with $0 < \mu(B_i) < 3^{-i}$.
Set $B := \bigcup_{i \in \bZ^+} B_i$, so that $0 < \mu(B) < 1/2$. 
Since each element of $(a_i)_{i \in \bZ^+}$ appears infinitely many times, that is $a_i = a_{i_k}$ for $i_k \to +\infty$, each function $\|\zeta(a_i)\|$ has essential supremum over $B$ of $\|\sigma(a_i)\|$. 
The decomposition $\widehat{\cA} = B \cup B^c$ splits $\sigma \cong \sigma_1 \oplus \sigma_2$ into a direct sum of two proper subrepresentations. 
By construction, the map $\sigma(a) \mapsto \sigma_1(a)$ is isometric on each $a_i$, and thus is isometric on $\cA_{sa}$, whence also on $\cA$.
This is a $*$-monomorphism, so it is also completely isometric.
\end{proof}

With a bit more care, one can make the map $\sigma(a) \mapsto \sigma_2(a)$ isometric as well, but we will not need this. We also note that $\cA$ need not be unital.

A unitary equivalence class $[\pi] \in \widehat{\cA}$ is called \textit{isolated} if $\{[\pi]\}$ is open in $\widehat{\cA}$. 
We caution the reader that even if $\cA$ is GCR, the topology need not be Hausdorff or even $T_1$. 
An isolated point $[\pi]$ may be such that $\{[\pi]\}$ is open but not closed, and this is not uncommon.

%%%%%%%%%%%%%%%%%%%%%%%%%%%%
\begin{lemma} \label{lem:open points}
Let $\cS$ be a separable operator system inside its $C^*$-envelope $\cA = C^*_e(\cS)$, and let $\pi \in \text{Irr}(\cA)$. 
Then $\pi$ is strongly peaking for $\cS$ if and only if $[\pi]$ is an isolated point of $\widehat{\cA}$.
\end{lemma}

\begin{proof}
First, suppose that $\pi$ is strongly peaking with respect to $\cS$. By Corollary~\ref{cor:peaking_projection}, there is an element $b \in\cA_{sa}$ such that $\|\pi(b)\|=1$ but $b \in \ker \sigma$ for any irreducible representation $\sigma \not\cong \pi$. By definition of the topology, $[\pi]$ does not belong to $\ol{\widehat{\cA} \setminus\{[\pi]\}}$.
That is, $\widehat{\cA} \setminus\{[\pi]\}$ is closed, so $\{[\pi]\}$ is open, i.e.\ $[\pi]$ is isolated.

Conversely, if $[\pi]$ is an isolated point, then  $\widehat{\cA} \setminus\{[\pi]\}$ is closed.
Hence $\pi$ does not vanish on $J:= \bigcap\limits_{[\sigma]\in \widehat{\cA} \setminus\{[\pi]\}} \ker\sigma$. 
Since $J$ is a nonzero ideal of $\cA= C^*_e(\cS)$, the quotient map by $J$ is not completely isometric on $\cS$. 
We conclude there is a matrix $\cS \in M_n(\cS)$ with $\|S\| > \|S + M_n(J) \|$. 
Any $\sigma\in\Irr(\cA)$ with $\sigma\not\cong\pi$ factors through $\cA/J$, so
\[
 \|S\| = \max \Big\{ \|\pi(S)\|, \!\!\sup_{\substack{\sigma\in\Irr(\cA)\\ \sigma\not\cong \pi}} \|\sigma(S)\| \Big\} 
 >  \|S + M_n(J) \| \ge\! \sup_{\substack{\sigma\in\Irr(\cA)\\ \sigma\not\cong \pi}} \|\sigma(S)\| .
\]
Therefore $\pi(S) = \|S\|$, and $\pi$ is strongly peaking for $\cS$.
\end{proof}

The lemma applies in particular to any separable unital $C^*$-algebra $\cA$, where we may take $\cS = \cA = C^*_e(\cS)$. 
In this case, one need only use elements of $\cA$ instead of matrices over $\cA$, since isometric $*$-homomorphisms are completely isometric. 

%%%%%%%%%%%%%%%%%%%%%%%%%%%%%
\begin{remark}
Once we know that $\pi$ is strongly peaking, Lemma~\ref{lem:peaking and compacts} applies to show that $\pi(\cA)$ contains $K(H_\pi)$, so $\pi$ lives on the maximal GCR ideal of $\cA$. In fact, there is an ideal $J$ of $\cA$ isomorphic to $K(H_\pi)$. If $\cA \subset B(H)$ is such that $\id$ is multiplicity free, then the concrete ideal $J$ is a summand of $\cA \cap K(H)$ that is unitarily equivalent to $K(H_\pi)$. This can be arranged, as every $C^*$-algebra admits a faithful multiplicity free representation.
\end{remark}

This remark has a converse.

%%%%%%%%%%%%%%%%%%%%%%%%%%%%
\begin{corollary} \label{cor:open points}
Let $\cS$ be a separable operator system inside its $C^*$-envelope $\cA = C^*_e(\cS)$, and let $\pi \in \text{Irr}(\cA)$. 
If there is an ideal $J$ of $\cA$ such that $J$ is isomorphic to $K(H_1)$ for some Hilbert space $H_1$ and $\pi|_J \ne 0$, then $\pi$ is strongly peaking. 
\end{corollary}

\begin{proof}
Let $\sigma \in \Irr(\cA)$. If $\sigma|_J \ne 0$, then because $K(H_1)$ has a unique irreducible representation, we have
\[ \sigma|_J \cong \id_{K(H_1)} \cong \pi|_J\]
and hence $\sigma\cong\pi$. The contrapositive gives that if $\sigma\not\cong\pi$, then $J \subseteq \ker\sigma$. Therefore,
\[ \{ [\sigma] \in \widehat{\cA}: \sigma \not\cong \pi\} = \{  [\sigma] \in \widehat{\cA}: J \subseteq \ker\sigma \} \]
is a closed subset of $\widehat{\cA}$. We have shown $[\pi]$ is an isolated point, and since $\cA = C^*_e(\cS)$, Lemma~\ref{lem:open points} implies $\pi$ is strongly peaking.
\end{proof}

Every $C^*$-algebra $\cA$ contains a largest GCR ideal $I$, as in \cite[\S 1.5]{ArvBook}. See also \cite[\S 4.3]{DixmierBook}, where GCR $C^*$-algebras are called postliminal. 
If $\pi$ is an irreducible representation of $\cA$ which does not annihilate $I$, then $\pi(\cA) \supseteq \pi(I) \supseteq K(H_\pi)$ contains all compact operators.
On the other hand, if $\ker \pi$ contains $I$, than $\pi$ factors through the NGCR (or antiliminal) $C^*$-algebra $\cA/I$, and $\pi(\cA)$ contains no nonzero compact operators.
A deep result of Glimm's \cite[Theorems 1 and 2]{Glimm61} characterizes GCR algebras as the type I $C^*$-algebras and NGCR algebras as having a very complicated representation theory.

Example \ref{ex:Cuntz-Toeplitz} shows there may be isolated points of $\widehat{C^*_e(\cS)}$ even if $C^*_e(\cS)$ is not GCR. 
However, if it is GCR, then we may connect fully compressed operator systems to minimal ones.
Indeed, it is enough that the maximal GCR ideal be essential, meaning that it has nontrivial intersection with every nonzero ideal.

%%%%%%%%%%%%%%%%%%%%%%%%%%%%
\begin{theorem}\label{thm:type I}
Let $\cS \subset B(H)$ be a separable operator system. Then the following are equivalent.
\begin{enumerate}
\item\label{item:fc_snooze} $\cS$ is fully compressed.
\item\label{item:isolated_list} $C^*(\cS) = C^*_e(\cS)$ and $\id$ is the direct sum of inequivalent irreducible representations $\pi_i$, where each class $[\pi_i]$ is isolated in $\widehat{C^*(\cS)}$.
\item\label{item:compact_min_cond} $\cS$ is minimal, $C^*(\cS) = C^*_e(\cS)$, and $C^*(\cS) \cap K(H)$ is essential.
\item\label{item:GCR_min_cond} $\cS$ is minimal, $C^*(\cS) = C^*_e(\cS)$, and the maximal GCR ideal $I$ is essential.
\end{enumerate}
\end{theorem}

\begin{proof} 
By Lemma \ref{lem:open points}, (\ref{item:isolated_list}) is equivalent to condition (\ref{item:individual_classification}) of Theorem \ref{thm:fc_full_equiv_general}, so that theorem shows that (\ref{item:fc_snooze}) and (\ref{item:isolated_list}) are equivalent.
Moreover they clearly imply
that $\cS$ is minimal and $C^*(\cS) = C^*_e(\cS)$. 
Since $\id$ is the direct sum of inequivalent strongly peaking representations $\pi_i$, Lemma~\ref{lem:peaking and compacts} shows that each $\pi_i$
is induced by its restriction to an irreducible summand of $C^*(\cS)\cap K(H)$ isomorphic to $K(H_{\pi_i})$. The $H_{\pi_i}$ span $H$, so it follows that for any nonzero element $a \in C^*(\cS)$, there exist some $i$ and some $p \in K(H_{\pi_i}) \subseteq C^*(\cS) \cap K(H)$ such that $pa \not= 0$. If $a$ belongs to some nonzero ideal $J$, then $pa \in J \cap (C^*(\cS) \cap K(H))$, so $C^*(\cS) \cap K(H)$ is an essential ideal. That is, (\ref{item:compact_min_cond}) holds. Now, (\ref{item:compact_min_cond}) trivially implies (\ref{item:GCR_min_cond}) since $I$ contains $C^*(\cS)\cap K(H)$.

Finally, assume (\ref{item:GCR_min_cond}).
We first show that $\id$ is induced by its restriction to $I$.
By \cite[Lemma 2.11.1]{DixmierBook}, $\id \cong \sigma_a \oplus \sigma_s$ where $\sigma_a$ is induced by $I$ and $\sigma_s$ annihilates $I$.
Thus $\sigma_a$ is faithful on $I$.
Since $I$ is essential, $\sigma_a$ is also faithful on $C^*(\cS)$.
Minimality of $\cS$ shows that the representation $\sigma_s$ must be vacuous, and $\id\cong\sigma_a$ is induced by its restriction to $I$.

Since $I$ is GCR, every separable representation of $I$, in particular the identity representation, is given by a direct integral over its spectrum. If there is a nonatomic part $\sigma$, then by Lemma~\ref{lem:direct integral}, there is a proper summand $\sigma_1$ of $\sigma$ which is completely isometric. It follows that the direct sum of the atoms and $\sigma_1$ is faithful on $I$. This is a compression of $I$ to a reducing subspace, and since $I$ is an essential ideal, the compression to the same reducing subspace of $C^*(\cS)$ is also faithful. This implies $\cS$ is not minimal, so the direct integral for $\id|_I$ must actually be a direct sum. That is, the identity on $C^*(\cS)$ is a direct sum of irreducible representations $\pi_i$ which are nonzero on $I$. 

Since $I$ is GCR and $\pi_i(I) \not= \{0\}$, the range of $\pi_i$ includes the compact operators. Noting that $C^*(\cS) = C^*_e(\cS)$ and $\cS$ is minimal, we conclude from Corollary~\ref{cor:envelope and minimal} that each irreducible summand $\pi_i$ of $\id$ is strongly peaking. 
Minimality of $\cS$ shows there is no multiplicity in the direct sum, so by Theorem \ref{thm:fc_full_equiv_general}, $\cS$ is fully compressed.
\end{proof}

The following result provides a necessary and sufficient condition to produce a fully compressed presentation of $\cS$, answering \cite[Question 4.19]{PasSha18}. Later, we will also consider some special cases in the language of matrix convexity.

%%%%%%%%%%%%%%%%%%%%%%%%%%%%
\begin{corollary}\label{cor:exists_fc}
Let $\cS$ be a separable abstract operator system. Then the following are equivalent.
\begin{enumerate}
\item\label{item:exists_fc_basicb} There exists a concrete presentation of $\cS$ which is fully compressed.
\item\label{item:exists_fc_density} The set of isolated points of $\widehat{C^*_e(\cS)}$ is dense.
\item\label{item:exists_fc_concrete} There exists a concrete presentation of $\cS$ such that $\cS$ is completely normed by irreducible representations of $C^*(\cS)$ which are strongly peaking for $\cS$.
\end{enumerate}
\end{corollary}

\begin{proof}
\textbf{(\ref{item:exists_fc_basicb})} $\!\Rightarrow\!$ \textbf{(\ref{item:exists_fc_density})}: 
Theorem \ref{thm:type I} shows that if a concrete presentation of $\cS$ is fully compressed, then $C^*(\cS) = C^*_e(\cS)$ 
and the identity is the direct sum of representatives of isolated points in $\widehat{C^*_e(\cS)}$. 
In particular, the intersection of their kernels is $\{0\}$, which means that they are dense.

\textbf{(\ref{item:exists_fc_density})} $\!\Rightarrow\!$ \textbf{(\ref{item:exists_fc_concrete})}: 
If the isolated points $\{[\pi_i] : i \in J\}$ are dense in $\widehat{C^*_e(\cS)}$, we may norm $C^*_e(\cS)$ with their direct sum.
Hence $s \mapsto \bigoplus \pi_i(s)$ is completely isometric on $\cS$ and defines a concrete presentation which generates the $C^*$-envelope. By Lemma \ref{lem:open points}, each $\pi_i$ is strongly peaking for $\cS$. 

\textbf{(\ref{item:exists_fc_concrete})} $\!\Rightarrow\!$ \textbf{(\ref{item:exists_fc_basicb})}: 
Suppose that there is a concrete presentation of $\cS$ such that $\cS$ is completely normed by irreducible representations $\pi_i$ of $C^*(\cS)$ which are strongly peaking for $\cS$. Then the direct sum presentation, without multiplicity, is a (possibly distinct) presentation in which the summands $\pi_i$ remain strongly peaking. By Theorem \ref{thm:fc_full_equiv_general}, this presentation is fully compressed.
 \end{proof}

These conditions are met whenever $C^*(\cS)$ has countable spectrum.

%%%%%%%%%%%%%%%%%%%%%%%%%%%%
\begin{corollary} \label{cor:countable_nice}
Let $\cS \subset B(H)$ be a separable operator system such that $\widehat{C^*(\cS)}$ is countable. Then there is a completely isometric image $\iota(\cS)$ which is fully compressed. If in addition $C^*(\cS) = C^*_e(\cS)$, we may write $\iota(\cS) = P_G \cS|_G$ for some reducing subspace $G \subseteq H$.
\end{corollary}

\begin{proof}
If $C^*(\cS)$ has countable spectrum, then so does the quotient $C^*_e(\cS)$, so we need only consider the case when $C^*(\cS) = C^*_e(\cS)$. 
By \cite[Lemma 3.2]{Jensen77}, $C^*(\cS)$ is GCR. 
The statement and proof of \cite[Theorem 4.4.5]{DixmierBook} then show that there is a dense open subset $X$ of $\widehat{C^*(\cS)}$ which is locally compact and Hausdorff. Now, $X$ is also countable, so the Baire Category Theorem implies that the isolated points of $X$ are dense in $X$. Since $X$ is open and dense, its isolated points remain isolated and dense in $\widehat{C^*(\cS)} = \widehat{C^*_e(\cS)}$.
In light of Lemma \ref{lem:open points}, $\cS$ is completely normed by its strongly peaking representations. 
Finally, Lemma \ref{lem:peaking and compacts} shows each strongly peaking representation of $C^*_e(\cS)$ is a subrepresentation of the identity, 
so we may let $G$ be the direct sum of the corresponding subspaces of $H$.
\end{proof}

In the above corollary, one must assume $C^*(\cS) = C^*_e(\cS)$ to conclude existence of the subspace $G$. As in \cite[Example 3.22]{Pas19}, if $T = \bigoplus \frac{1}{n}$, there is no minimal compression of $\cS_T$, so there is certainly not a fully compressed one. Note that $C^*(\cS_T) \not= C^*_e(\cS_T)$ has many isolated points in its spectrum which are not strongly peaking for $\cS_T$, as well as one point which is not strongly peaking for $\cS_T$ even though it corresponds to a strongly peaking representation of the $C^*$-envelope $C(\{0, 1\})$. The corollary also generalizes \cite[Corollary 6.8]{DDSS}, noting that the nonsingularity assumption therein is designed so that a compression to some reducing subspace will generate the $C^*$-envelope. 

On the other hand, we note that even if $C^*(\cS) = C^*_e(\cS)$ is commutative, it is possible for the spectrum to be uncountable yet still have a dense collection of isolated points.

%%%%%%%%%%%%%%%%%%%%%%%%%%%%
\begin{example}
There exists a separable operator system $\cS \subset B(H)$ which is fully compressed, but such that $C^*_e(\cS)$ has uncountably many irreducible representations, all of which are boundary representations of $\cS$.

Let $\theta \in \mathbb{R} \setminus \mathbb{Q}$, and define
\[
 P := \{(1 - 1/n)e^{2 \pi i n \theta}: n \in \bZ^+ \}
 \qand
 X := \overline{P} = P \cup \mathbb{S}^1 .
\]
Then $P$ is exactly the set of isolated points of $X$. View $C(X)$ as an operator system $\cS$ with the concrete presentation $f \mapsto \diag\{ f(p) : p \in P\} $ in $B(\ell^2(P))$. 
Since $\cS$ is a commutative $C^*$-algebra, the boundary representations are the irreducible representations $\delta_x$ for $x \in X$.
The strongly peaking representations are point evaluations at isolated points, namely $\delta_p$ for $p \in P$. 
The concrete presentation given above is fully compressed in $B(\ell^2(P))$, yet there are uncountably many boundary representations. 
\end{example}
%%%%%%%%%%%%%%%%%%%%%%%%%%%%%

We next show that if $C^*(\cS)$ is assumed block diagonal, then $\cS$ is fully compressed if and only if it is minimal. 

%%%%%%%%%%%%%%%%%%%%%%%%%%%%
\begin{theorem}\label{thm:block_diagonal}
Suppose $\cS \subset B(H)$ is a separable operator system such that $C^*(\cS)$ is block diagonal, that is, $\text{id} \cong \bigoplus\limits_{i \in \Lambda} \psi_i$ for finite-dimensional, irreducible representations $\psi_i: C^*(\cS) \to M_{n_i}(\bC)$. Then the following are equivalent.
\begin{enumerate}
\item\label{item:bd_sep_min} $\cS$ is minimal.
\item\label{item:bd_sep_single} For each $i \in \Lambda$, the map $s \mapsto \bigoplus\limits_{j \in \Lambda \setminus \{i\}} \psi_j(s)$ is not completely isometric on $\cS$.
\item\label{item:bd_sep_sp} The $\psi_i$ are strongly peaking representations, enumerated once per unitary equivalence class.
\item\label{item:bd_sep_fc} $\cS$ is fully compressed.
\end{enumerate}
Moreover, in this case $C^*(\cS) = C^*_e(\cS)$.
\end{theorem}

\begin{proof}
The implications (\ref{item:bd_sep_fc}) $\Rightarrow$ (\ref{item:bd_sep_min}) $\Leftrightarrow$ (\ref{item:bd_sep_single}) are trivial.
By Theorem \ref{thm:fc_full_equiv_general},  (\ref{item:bd_sep_sp}) $\Leftrightarrow$ (\ref{item:bd_sep_fc}). 
In particular, since it is possible to completely norm $\cS$ using finite-dimensional irreducible representations, all the strongly peaking representations are finite-dimensional. 

$\textbf{(\ref{item:bd_sep_single})} \!\Rightarrow\! \textbf{(\ref{item:bd_sep_sp})}$: Write $H \cong \bigoplus\limits_{j \in \Lambda} H_j$, where $\psi_j$ acts on the finite-dimensional space $H_j$, and fix $i \in \Lambda$. 
Since the map $s \mapsto \bigoplus\limits_{j \in \Lambda \setminus \{i\}} \psi_j(s)$ is not completely isometric, there is a matrix $S \in M_n(\cS)$ for which $\|S\| = \| \psi_i (S) \| > \sup\limits_{j \not= i} \| \psi_j(S) \|$. From \cite[Lemma 3.3]{DavKen15}, there is a pure matrix state $\phi: \cS \to M_n(\bC)$ such that $\| \phi(S)\| > \sup\limits_{j \not= i} \| \psi_j(S) \|$. Moreover, \cite[Theorem B]{Far00} and \cite[Theorem 4.6]{WW99} show there exist isometries $V_k: \bC^n \to H_{j_k}$ such that 
\[ \phi(s) = \lim\limits_{k \to \infty} V_k^* \psi_{j_k}(s) V_k \qforal s \in \cS. \]

If $j_k \not= i$ for infinitely many $k$, we reach a contradiction by plugging in the matrix $S$, so we may assume $j_k = i$ for all $k$. Since $\bC^n$ and $H_{i}$ are finite-dimensional, the $V_k$ have a norm convergent subsequence, and hence $\phi$ is a compression of $\psi_i$. Setting $\phi_0 := \phi$, we may form a dilation $\phi_0 \prec \phi_1 \prec \ldots$ through pure matrix states to reach an accessible boundary representation by \cite[Remark 2.5 and Lemma 3.3]{DavKen15}, and the same argument shows each $\phi_m$ is a compression of $\psi_i$. Since $\psi_i$ acts on a finite-dimensional space, this dilation must terminate, so the largest $\phi_m$ is pure and maximal. Maximality of $\phi_m$ and irreducibility of $\psi_i$ imply that $\psi_i \cong \phi_m$, so $\psi_i$ is boundary and consequently $C^*(\cS) = C^*_e(\cS)$.  

The compression to the cofinite-dimensional subspace $G = H_i^\perp$  is not completely isometric, so the map $P_G s|_G \mapsto \psi_i(s)$ is not UCP. 
By Lemma \ref{lem:detection_dichotomy},  $\psi_i$ does not annihilate $C^*(\cS) \cap K(H)$. 
Since $\psi_i$ factors through the $C^*$-envelope, Corollary~\ref{cor:envelope and minimal} also shows $\psi_i$ is strongly peaking.
By definition of strongly peaking, since the $\psi_i$ are sufficient to completely norm $\cS$, every strongly peaking representation must be unitarily equivalent to some $\psi_i$. The assumptions directly rule out multiplicity, so $\cS$ is fully compressed by Theorem \ref{thm:fc_full_equiv_general}.
\end{proof}

%%%%%%%%%%%%%%%%%%%%%%%%%%%%
\begin{corollary}
If two separable operator systems are minimal, completely order isomorphic, and block diagonal, then any complete order isomorphism between them is a unitary equivalence. Similarly, if $T \in B(H_1)^d$ and $S \in B(H_2)^d$ are block diagonal, and both $S$ and $T$ are minimal with $\cW(T) = \cW(S)$, then $T \cong S$.
\end{corollary}

Theorem \ref{thm:block_diagonal} applies whether there are finitely many or infinitely many summands, so it generalizes finite-dimensional uniqueness results such as \cite[Theorem 1.2]{HKM13} and \cite[Theorem 1.2]{EHKM16}. See also the discussion before Corollary \ref{cor:finite_dim_cons}.

In \cite{PasSha18}, certain special cases of fully compressed $d$-tuples were classified without using strongly peaking representations, but rather a different concept known as a \textit{crucial matrix extreme point}. Our classification of fully compressed operator systems (or $d$-tuples) subsumes these results, so we now address how these two notions are related. Since crucial matrix extreme points are a reflection of the abstract operator system $\cS_T$, while strongly peaking representations rely on the concrete $C^*$-algebra in which an operator system lives, there are some caveats along the way. First, recall the following definitions.

%%%%%%%%%%%%%%%%%%%%%%%%%%%%
\begin{definition}
A \textit{matrix convex set} over $\mathbb{C}^d$ is a set $\cC \subseteq \bigcup\limits_{n=1}^\infty M_n(\bC)^d$, written $\cC = \bigcup\limits_{n=1}^\infty \cC_n$, such that the following conditions hold.
\begin{itemize}
\item If $X \in \cC_n$ and $Y \in \cC_m$, then $X \oplus Y \in \cC_{n + m}$.
\item If $X \in \cC_n$ and $\phi: M_n(\bC) \to M_m(\bC)$ is UCP, then $\phi(X) \in \cC_m$.
\end{itemize}
\end{definition}

Matrix convex sets are more generally defined over arbitrary vector spaces \cite[Definition 1.1]{WW99}, but we will only apply results in the above context. Also, because of Stinespring's theorem (see \cite[Theorem 4.1]{PauBook}), it suffices to replace the second condition with the weaker assumption that $\cC$ is closed under compressions to subspaces. That is, if $X \in \cC_n$ and $\alpha: \bC^m \to \bC^n$ is an isometry ($m \le n$), then $\alpha^* X \alpha \in \cC_m$. This implicitly includes the fact that matrix convex sets are closed under unitary conjugation. If $\alpha$ is a proper isometry, we say that $\alpha^* X \alpha$ is a proper compression of $X$.

A matrix convex set is called closed if each level $\cC_n = \cC \cap M_n(\bC)^d$ is closed, and bounded if there is a uniform matrix norm bound on any element of $\cC$. However, \cite[Corollary 4.6]{DDSS} implies boundedness of $\cC$ is equivalent to boundedness of the first level $\cC_1$. Closed and bounded matrix convex sets over $\bC^d$ are also precisely the matrix ranges $\cC = \cW(T)$ for $T \in B(H)^d$ by \cite[Propositions 2.5 and 3.5]{DDSS}. 

Matrix convex sets have two well-established definitions of extreme point, which we give below. Without the presence of infinite-dimensional points, there are some complications in describing the extreme points and Krein-Milman type results. See \cite[\S 6]{DavKen19} for an alternative notion of \textit{nc convexity}, which resolves some of these issues by including infinite-dimensional levels.

%%%%%%%%%%%%%%%%%%%%%%%%%%%%
\begin{definition} Let $\cC$ be a closed and bounded matrix convex set over $\bC^d$, and let $Y \in \cC_n$. 
Then $Y$ is called a \textit{matrix extreme point} if whenever
\[ 
Y = V_1^* \, X^{(1)} \, V_1 + \ldots + V_m  \, X^{(m)}  \, V_m 
\]
for rectangular matrices $V_i$ with $\sum V_i^* V_i = I$ and elements $X^{(i)} \in \cC_{n_i}$ with $n_i \leq n$, it follows that each $X^{(i)}$ is unitarily equivalent to $Y$.

Further, $Y$ is called an \textit{absolute extreme point} if whenever
\[
 Y = V_1^* \, X^{(1)} \, V_1 + \ldots + V_m  \, X^{(m)}  \, V_m 
\]
for rectangular matrices $V_i$ with $\sum V_i^* V_i = I$ and elements $X^{(i)} \in \cC$, it follows that each $X^{(i)}$ is unitarily equivalent to $Y$ or to $Y\oplus Z_i$ for some $Z_i \in \cC$.
\end{definition}

A matrix extreme point might still be written as a nontrivial combination of elements from higher levels of $\cC$, and the matrix extreme points of $\cW(T)$ are precisely the images of $T$ under pure matrix states by \cite[Theorem B]{Far00}. Finally, matrix extreme points enjoy a Krein-Milman theorem \cite[Theorem 4.3]{WW99} and an associated Milman converse \cite[Theorem 4.6]{WW99}. 

The absolute extreme points are much more restrictive, as seen in \cite[Corollary 1.1]{Eve18}. The results \cite[Theorem~4.2]{Kleski} and \cite[Lemma~6.10 and Corollary~6.28]{Kriel} show that the absolute extreme points are irreducible elements of $\cC$ which have no nontrivial dilations. That is, if an absolute extreme point $Y$ of $\cC$ is a proper compression of some $X \in \cC$, then $X \cong Y \oplus Z$ for some $Z \in \cC$. Such \textit{maximal} extreme points have the unique extension property by \cite[Proposition 2.4]{Arv08}, hence absolute extreme points correspond to finite-dimensional boundary representations. These facts admit generalizations to nc convex sets, as in \cite[\S 6]{DavKen19}.

From this point of view, an isolated extreme point is generalized by some $Y \in \cC$ which cannot be generated by the matrix extreme points that are inequivalent to $Y$.

%%%%%%%%%%%%%%%%%%%%%%%%%%%%%
\begin{definition}\cite[Definition 2.4]{PasSha18}\label{def:crucial_MEP}
Let $Y$ be an element of a closed and bounded matrix convex set $\cC$ over $\bC^d$, and define 
\[
 \cE_Y := \{X \in \cC: X \text{ is matrix extreme}, X \not\cong Y\}.
\] 
Then $Y$ is called a \textit{crucial matrix extreme} point of $\cC$ if the closed matrix convex hull of $\cE_Y$ does not include $Y$.
\end{definition}
\begin{remark}
The definition does not explicitly require $Y$ to be matrix extreme, but this follows from the Krein-Milman theorem \cite[Theorem 4.3]{WW99}. In fact,  \cite[Proposition 2.6]{PasSha18} implies every crucial matrix extreme point is absolute extreme, i.e. it corresponds to a finite-dimensional boundary representation.
\end{remark}

The following example, which is based upon \cite[Example 3.14]{Pas19}, shows that crucial matrix extreme points and finite-dimensional strongly peaking representations are not always identified.

%%%%%%%%%%%%%%%%%%%%%%%%%%%%
\begin{example}\label{ex:nonpeaking_scum}
Let $\Delta$ be the convex hull of $(0, 0), (1, 0)$, and $(0, 1)$. Since $\Delta$ is a simplex, it is the first level of a unique matrix convex set by \cite[Theorem 4.1]{PSS18} or \cite[Theorem 4.7]{FNT}. Let $T = (1, 0) \oplus (0, 1) \oplus (K_1, K_2)$, where the last summand is an infinite-dimensional, irreducible pair of positive compact operators with $ K_1 + K_2 \leq \frac{1}{2} I$. Then $\cW(T)$ is the unique matrix convex set over $\Delta$, and $\cS_T \subseteq K(H) + \mathbb{C} \, I$.

The crucial matrix extreme points of $\cW(T)$ are $(0,0)$, $(1, 0)$, and $(0, 1)$ by \cite[Proposition 2.5]{PasSha18}, since $\cW(T)$ is the minimal matrix convex set over $\Delta$. The point $(0, 0)$ is detected by the singular state $\pi_\infty: C^*(\cS_T) \to \bC$ which annihilates the compacts. This representation is a limit of compressions of $\pi_K: T \mapsto K$, which is a subrepresentation of the identity. We conclude $\pi_\infty$ is not strongly peaking for $\cS_T$. Similarly, $\pi_K$ is not strongly peaking since $\cS_T$ is completely normed by one-dimensional representations.

The pair $(T_1,T_2)$ is minimal but not fully compressed, and $C^*(\cS_T)$ is much larger than the $C^*$-envelope. In fact, the assumption of nonsingularity in \cite[\S 6]{DDSS} was specifically designed to address issues with $\pi_\infty$ for the uniqueness of minimal compact tuples, based upon this example.  In the context of Theorem \ref{thm:fc_full_equiv_general}, for a tuple $T \in K(H)^d$ of infinite-dimensional compacts, the singular state $\pi_\infty$ will never appear as a strongly peaking representation. 
\end{example}

As long as one restricts attention to the $C^*$-envelope, the disparity between crucial matrix extreme points and finite-dimensional strongly peaking representations disappears. To show this, we will need the following separation theorem. It is an immediate consequence of \cite[Theorem 1.6]{WW99}, which is itself a variant of results in \cite[\S 5]{EW97}. A translation from $I$ to $0$ is possible because all UCP maps are unital. While we state the claim in full generality, we note that it reduces to consideration of the special case $\cS = \cS_T$, or rather matrix convex sets over $\bC^d$, by a straightforward argument.

%%%%%%%%%%%%%%%%%%%%%%%%%%%%
\begin{theorem} \label{thm:EW_variant}
Let $\bK$ be the matrix state space of an operator system $\cS$.
Suppose that $\bL$ is a proper, pointwise closed, matrix convex subset of $\bK$, and let $x \in K_n \setminus L_n$.
Then there is a matrix $S=S^* \in M_n(\cS)$ such that
\[
  y(S) \le 0 \FORAL y \in \bL \qand  x(S) \not\le 0  .
\]
\end{theorem}

%%%%%%%%%%%%%%%%%%%%%%%%%%%%
\begin{theorem}\label{thm:no_prob_envelope}
Fix $T \in B(H)^d$ and the concrete operator system $\cS_T$. If $C^*(\cS_T) = C^*_e(\cS_T)$, then the map $\pi \mapsto \pi(T)$ is a bijection between finite-dimensional strongly peaking representations of $\cS_T$ and crucial matrix extreme points of $\cW(T)$.

If $C^*(\cS_T) \ne C^*_e(\cS_T)$, then the map $\pi \mapsto \pi(T)$ still sends finite-dimensional strongly peaking representations
to crucial matrix extreme points of $\cW(T)$, but it need not be surjective.
\end{theorem}

\begin{proof}
Suppose there is a strongly peaking representation $\pi: C^*(\cS_T) \to M_n(\bC)$ with $\pi(T) = Y$, where we do not assume $C^*(\cS_T) = C^*_e(\cS_T)$. Since any strongly peaking representation $\pi$ is pure, $Y$ is a matrix extreme point by \cite[Theorem B]{Far00}. If $Y$ is not crucial matrix extreme, then the closed matrix convex hull of
\[ \cE_Y = \{X \in \cW(T): X \text{ is matrix extreme}, \, X \not\cong Y\}\]
includes $Y$, hence equals $\cW(T)$ by the matricial Krein-Milman theorem of Webster and Winkler \cite[Theorem 4.3]{WW99}. 
Moreover, their converse \cite[Theorem 4.6]{WW99} implies that the closure of $\cE_Y$ together with all compressions to smaller subspaces must contain the matrix extreme point $Y$.
Therefore, $Y$ is a limit of compressions of some sequence $Z_k \in \cE_Y$. In particular, each $Z_k$ has $\dim(Z_k) \geq \dim(Y)$.
The pure matrix state $\phi_k: T \mapsto Z_k$ is either already maximal, and hence it is a finite-dimensional boundary representation $\sigma_k$, or it dilates to an accessible boundary representation $\sigma_k$ on some higher dimensional space (possibly infinite). Since $Z_k \in \cE_Y$ and $\dim(Z_k) \geq \dim(Y)$, in either case we have $\sigma_k \not\cong \pi$. We conclude that, on $\cS_T$, $\pi$ is a pointwise limit of compressions of irreducible representations inequivalent to $\pi$. This contradicts the fact that $\pi$ is strongly peaking, so we must have that $Y$ is crucial matrix extreme.

Suppose conversely that  $Y$ is a crucial matrix extreme point. We seek a strongly peaking representation of the $C^*$-envelope.
Since $Y$ is absolute extreme \cite[Proposition 2.6]{PasSha18}, there is a finite-dimensional boundary representation $\pi$ with $\pi(T) = Y$ by \cite[Theorem 4.2]{Kleski}. 
Theorem~\ref{thm:EW_variant} shows there is some $S=S^* \in M_n(\cS_T)$ with
\[
 \phi(S) \le 0 \text{ for all  } \phi\in \cE_Y \qand \pi(S) \not\le 0 .
\]

Fix a concrete presentation $C^*(\iota(\cS_T))$ of the $C^*$-envelope $C^*_e(\cS_T)$ by taking $\iota = \pi \oplus \bigoplus\limits_{\sigma \in \Omega} \sigma$, where $\Omega$ is a sufficient countable collection of accessible boundary representations inequivalent to $\pi$.
Let $f(x) = \max\{ 0, x \}$ for $x \in \bR$.
If $\sigma \in \Omega$, then $\sigma$ is an accessible boundary representation which is either in $\cE_Y$ or a dilation of points $\phi_k\in\cE_Y$,
so $\sigma(S) = \lim \phi_k(S) \le 0$. Therefore $\sigma(f(S)) = 0$ while $\pi(f(S)) \ne 0$.

We have shown that the compression of $C^*_e(\cS_T) = C^*(\iota(\cS_T))$ to the reducing subspace $H_\pi^\perp$ is not isometric. Since we are on the $C^*$-envelope, the compression is not completely isometric on $\cS_T$. Certainly the range of $\pi$ contains the compacts since it is finite-dimensional, so Corollary \ref{cor:envelope and minimal} implies that $\pi$ is strongly peaking.

Finally, Example \ref{ex:nonpeaking_scum} shows that if $C^*(\cS_T) \ne C^*_e(\cS_T)$, surjectivity may fail.
\end{proof}

The above claims apply in particular to the matrix case, i.e. to $\cW(A)$ where $A$ is a tuple of matrices. Note that free spectrahedra and matrix ranges of matrix tuples are related via the polar dual, as in \cite[Proposition 3.1]{DDSS}. Thus, \cite[Theorem 1.2]{EHKM16} implies that a matrix tuple $A$ is minimal whenever $A$ is the direct sum of absolute extreme points of $\cW(A)$, and vice versa. Since boundary representations correspond to absolute extreme points, this claim may be recovered by our results (and some results of Arveson). We give the following corollary for completeness.

%%%%%%%%%%%%%%%%%%%%%%%%%%%%
\begin{corollary}\label{cor:finite_dim_cons}
If $A$ is a tuple of matrices, then the crucial matrix extreme points of $\cW(A)$ are precisely the absolute extreme points of $\cW(A)$. In particular, there are only finitely many such points up to unitarily equivalence, each of which is a summand of $A$. Finally, $A$ is minimal if and only if $A$ is multiplicity-free and every summand of $A$ is a crucial matrix extreme point.
\end{corollary}

\begin{proof}
By \cite[Proposition 2.6]{PasSha18}, every crucial matrix extreme point is absolute extreme. If $X \in \cW(A)$ is absolute extreme, then there is a boundary representation $\phi$ of $\cS_A$ which maps $A$ to $X$ by \cite[Corollary 6.28]{Kriel}. This representation is in particular unital, so finite-dimensionality of $A$ implies that $\phi$ does not annihilate the compacts. Thus, \cite[Theorem 7.2]{Arv11} shows that $\phi$ is strongly peaking for $\cS_A$. We may then apply Theorem \ref{thm:no_prob_envelope} (whether or not $\cS_A$ sits inside its $C^*$-envelope) to see that $X$ is crucial matrix extreme. The remainder of the proposition follows from Theorems \ref{thm:block_diagonal} and \ref{thm:no_prob_envelope}.
\end{proof}

Returning to the case of infinite-dimensional operators, we next see that whenever $\cC = \cW(T)$ is generated by its absolute extreme points and $T$ is fully compressed, we may conclude block diagonality of $T$ instead of assuming it.

%%%%%%%%%%%%%%%%%%%%%%%%%%%%
\begin{corollary}\label{cor:AEP_generated_uniqueness}
Let $T \in B(H)^d$, and suppose $\cW(T)$ is the closed matrix convex hull of its absolute extreme points. Then $T$ is fully compressed if and only if $T \cong \bigoplus\limits_{X \in \Lambda} X$, where $\Lambda$ is precisely the set of crucial matrix extreme points, enumerated once per unitary equivalence class.
\end{corollary}

\begin{proof}
Since $\cW(T)$ is the closed matrix convex hull of its absolute extreme points, $\cS_T$ is completely normed by finite-dimensional boundary representations, hence every strongly peaking representation is finite-dimensional. Moreover, either claim implies that $C^*(\cS_T) = C^*_e(\cS_T)$. The result then follows from Theorems \ref{thm:fc_full_equiv_general} and \ref{thm:no_prob_envelope}.
\end{proof}

The assumptions of Corollary \ref{cor:AEP_generated_uniqueness} apply, for example, if $\cS_T$ sits inside an FDI $C^*$-algebra (see \cite[\S 1]{Courtney}). These $C^*$-algebras are used in \cite[Theorem 1.5]{Hartz} to prove finite-dimensional dilation results. Similarly, \cite[Corollary 6.13]{Kriel} shows that whenever a closed and bounded matrix convex set $\cC$ is generated by some fixed level $\cC_n$, it is generated by its absolute extreme points. However, \cite[Example 7.10 and Remark 7.11]{Kriel} show that generation by a finite level is not a necessary condition. That is, it is possible for an operator system to be completely normed by finite-dimensional boundary representations but still admit infinite-dimensional ones.

%%%%%%%%%%%%%%%%%%%%%%%%%%%%
\begin{corollary}
Let $\cC$ be a closed and bounded matrix convex set over $\bC^d$ which is the closed matrix convex hull of its absolute extreme points. Then there exists a fully compressed tuple $T \in B(H)^d$ with $\cW(T) = \cC$ if and only if $\cC$ is the closed matrix convex hull of its crucial matrix extreme points.
\end{corollary}

For matrix convex sets generated by the first level, i.e. sets of the form $\cW^{\text{min}}(K)$, the above results improve the relevant portion of \cite[Theorem 4.10]{PasSha18}. In this case, a fully compressed tuple for $\cW(T) = \Wmin{}(K)$ is the direct sum of isolated extreme points of $K$, without the need to assume $T$ is normal ahead of time.

%%%%%%%%%%%%%%%%%%%%%%%%%%%%
\begin{corollary}
Let $K$ be a compact convex subset of Euclidean space. Then there is a fully compressed tuple $T \in B(H)^d$ with $\cW(T) = \Wmin{}(K)$ if and only if the isolated extreme points of $K$ are dense in $\overline{\text{ext}(K)}$. In this case, $T$ is the direct sum of the isolated extreme points without multiplicity.
\end{corollary}

\begin{proof}
From \cite[Proposition 2.5]{PasSha18}, the crucial matrix extreme points of $\Wmin{}(K)$ are precisely the isolated extreme points of $K$.
\end{proof}

In particular, $\Wmin{}(\mathbb{D})$ cannot be written as the matrix range of a fully compressed tuple. 
Consistent with \cite[Proposition 7.5 and Corollary 8.3]{EHKM16}, it is definitely not the matrix range of a tuple of matrices. 
That is, the free spectrahedron $\Wmin{}(\mathbb{D}) = \cD_F$, 
where $F =  \left( \begin{bmatrix} 1 & 0 \\ 0 & -1\end{bmatrix}, \begin{bmatrix} 0 & 1 \\ 1 & 0 \end{bmatrix} \right)$, 
is not the polar dual of a free spectrahedron.

We close the section with discussion of how crucial matrix extreme points interface with the \textit{matrix exposed} points of \cite[Definition 6.1]{Kriel}. While matrix exposed points are defined generally, it causes no harm to assume that a matrix convex set consists of self-adjoints (by splitting into real/imaginary parts) and that $0$ is an interior point (by using an affine transformation, reducing the number of coordinates if necessary). That is, we may assume that
\[ 
 \cC \subseteq \mathbb{M}^d_{sa} = \bigcup\limits_{n=1}^\infty M_n(\bC)^d_{sa}
\]
and $0$ is an interior point of $\cC_1 \subseteq \bR^d$.

%%%%%%%%%%%%%%%%%%%%%%%%%%%%
\begin{definition}\cite[Definition 6.1]{Kriel}
Let $\cC \subseteq \mathbb{M}^d_{sa}$ be closed, bounded, and matrix convex. Then $Y \in \cC_n$ is \textit{matrix exposed} if there exist $A \in M_n(\bC)^d_{sa}$ and $B \in M_n(\bC)_{sa}$ such that
\[
 \cL := \Big\{X \in \mathbb{M}^d_{sa}: \sum\limits_{j=1}^d X_j \otimes A_j \leq I \otimes B \Big\} 
\]
has $\cC \subseteq \cL$ and $\cC_n \cap \partial \cL_n = \cU(Y)$, i.e. the unitary orbit of $Y$.
\end{definition}

If, in addition, $0$ is an interior point of $\cC_1$, we may assume that $B$ is the identity matrix, so that $\cL$ is the free spectrahedron 
\[ \cD_A := \Big\{X \in \mathbb{M}^d_{sa}: \sum\limits_{j=1}^d X_j \otimes A_j \leq I \Big\} .\]
In this case, we note that $\partial \cL_n = \partial \cD_A(n)$ consists of $X \in M_n(\bC)^d_{sa}$ such that $\sum_{j=1}^d X_j \otimes A_j \le I$ and this sum has $1$ as an eigenvalue.

%%%%%%%%%%%%%%%%%%%%%%%%%%%%%
\begin{remark}
If $Y \in \cC_1$, then $Y$ is matrix exposed in $\cC$ if and only if it is exposed in $\cC_1$. There is a different analogue of an exposed point in the first level, called an $\cS$-peaking state, in \cite[p.223]{Clo18}. However, these notions do not always agree, as the definition of an $\cS$-peaking state requires that the state extends uniquely from $\cS$ to $C^*(\cS)$, as in \cite[Theorem 1.1]{Clo18}. Similarly, \cite[Proposition 3.5]{Clo18} implies that there exist finite-dimensional operator systems $\cS$ which admit no $\cS$-peaking states, such as $\cS = \cS_T$ where $T = (T_1, T_2)$ is an irreducible pair of Cuntz isometries. This occurs even though $\cS_T$ generates the $C^*$-envelope.
\end{remark} 

The matrix dimension of $Y$, namely $n$, is used repeatedly in the definition of a matrix exposed point, just as in the definition of a matrix extreme point.  The results \cite[Proposition 6.19 and Theorem 6.21]{Kriel} further characterize matrix exposed points as matrix extreme points which are exposed in the level $\cC_n$ to which they belong. Crucial matrix extreme points, on the other hand, absolutely must see information in the higher levels, so it should not be possible to characterize them with an isolation condition solely on the unitary orbit in $\cC_n$. The following example demonstrates this.

%%%%%%%%%%%%%%%%%%%%%%%%%%%%
\begin{example}
Consider $\cC := \Wmax{}([-1, 1]^2)$, the largest matrix convex set whose first level is $[-1,1]^2$. Then $(1, 1)$ is a matrix/absolute extreme point of $\cC$, and it is in fact an isolated extreme point of $\cC_1$. However, $(1, 1)$ is not a crucial matrix extreme point. For $-1 < x < 1$, we have that $Y(x) := \left( \begin{bmatrix} 1 & 0 \\ 0 & -1 \end{bmatrix}, \begin{bmatrix} x & \sqrt{1 - x^2} \\ \sqrt{1 - x^2} & -x \end{bmatrix} \right)$ is matrix/absolute extreme, and $(1, 1)$ is a limit of compressions of $Y(x)$.
\end{example}

Exclusion of a point $Y$ from a closed matrix convex set is witnessed by Effros-Winkler separation \cite[Theorem 5.4]{EW97}. Thus, if $Y$ is crucial matrix extreme and $0$ is in the interior of the first level of $\cC \subset \mathbb{M}^d_{sa}$, then there exist some $A \in M_n(\bC)^d_{sa}$ and $\varepsilon > 0$ such that
\[
 X \in \cE_Y \,\, \implies \,\, \sum\limits_{j=1}^d A_j \otimes X_j \leq (1 - \varepsilon) I
\]
but $\sum\limits_{j=1}^d A_j \otimes Y_j \leq I$ has $1$ in its spectrum. Analogous to the classical setting for exposed and isolated extreme points, crucial matrix extreme points are matrix exposed, and any linear pencil witnessing this fact also demonstrates the Effros-Winkler separation.

%%%%%%%%%%%%%%%%%%%%%%%%%%%%
\begin{lemma}\label{lem:super_separating}
Let $\cC \subseteq \mathbb{M}^d_{sa}$ be a closed and bounded matrix convex set with $0$ in its interior, and assume $Y \in \cC_n$ is a crucial matrix extreme point of $\cC$. Then $Y$ is matrix exposed, and given any witness $A \in M_n(\bC)^d_{sa}$ such that $\cC \subseteq \cD_A$ and $\cC_n \cap \partial \cD_A(n) = \cU(Y)$, it follows that there exists $\varepsilon > 0$ with
\[
 X \in \cE_Y \hspace{.2 cm} \implies \hspace{.2 cm} \sum_{j=1}^d A_j \otimes X_j \leq (1 - \varepsilon) I. 
\]
\end{lemma}

\begin{proof}
 Assume $Y \in \cC_n$ is a crucial matrix extreme point of $\cC$. By \cite[Corollary 6.23]{Kriel}, $Y$ is in the closed matrix hull of the matrix exposed points, and these points are matrix extreme by \cite[Proposition 6.19]{Kriel}. By definition of a crucial matrix extreme point, the generating set must include a point in the unitary orbit of $Y$. That is, we may conclude $Y$ is itself matrix exposed. This yields a witness $A \in M_n(\bC)_{sa}$ such that $\cC \subseteq \cD_A$ and $\cC_n \cap \partial \cD_A(n)$ is precisely the unitary orbit of $Y$. 

With any choice of $A$ as above fixed, suppose there is a sequence $X^{(k)}$ of points in $\cE_Y$ such that the largest eigenvalue of the self-adjoint matrix $\sum\limits_{i=1}^d A_j \otimes X^{(k)}_j$ approaches $1$. Choosing an eigenvector $v^{(k)} = (v^{(k)}_1, \ldots, v^{(k)}_n) \in \bC^n \otimes \bC^{\text{dim}(X^{(k)})}$ for the largest eigenvalue and compressing $X^{(k)}$ to the span of the $v^{(k)}_j$ shows that there are $n \times n$ or smaller compressions $Z^{(k)}$ of $X^{(k)}$ such that the largest eigenvalue of  $\sum\limits_{j=1}^d A_j \otimes Z^{(k)}_j$ still approaches $1$.

We may choose a subsequence so that the matrix tuples $Z^{(k)}$ admit a limit $Z$ in some $\cC_m$, $m \leq n$. If $m < n$, then pick an arbitary $W \in \cC_{n-k}$ and note that $Z \oplus W \in \cC_n$ belongs to the boundary of $\mathcal{D}_A(n)$, so it is unitarily equivalent to $Y$. This contradicts the fact that matrix exposed points must be irreducible. If $m = n$, then we have that $Z$ itself is in the boundary of $\mathcal{D}_A(n)$, hence $Z \cong Y$. That is, $Y$ is a limit of compressions of the $X^{(k)}$, hence $Y$ is in the closed matrix convex hull of $\cE_Y$. This contradicts the fact that $Y$ is a crucial matrix extreme point. We conclude that there exists $\varepsilon > 0$ such that any $X \in \cE_Y$ satisfies $\sum\limits_{j=1}^d A_j \otimes X_j \leq (1 - \varepsilon) I$, as desired.
\end{proof}

 Our final result classifies crucial matrix extreme points in terms of 
\[
 \cF_Y :=  \{X \in \cC: X \text{ is matrix extreme, } X \text{ is not a compression of } Y\} 
\]
instead of $\cE_Y$, where we note as before that compressions include points in the unitary orbit. This result may be viewed through the lens of Theorem \ref{thm:no_prob_envelope} as a matrix convex condition which detects finite-dimensional strongly peaking representations of $C^*_e(\cS_T)$.

%%%%%%%%%%%%%%%%%%%%%%%%%%%%
\begin{theorem}\label{thm:new_crucial}
Let $\cC$ be a closed and bounded matrix convex set over $\bC^d$, and let $Y \in \cC$. Then the following are equivalent.

\begin{enumerate}
\item\label{item:blah} $Y$ is a crucial matrix extreme point. That is, $Y$ is not in the closed matrix convex hull of
\[
 \cE_Y = \{X \in \cC: X \text{ is matrix extreme}, \, X \not\cong  Y\}.
\]
\item\label{item:compression_sent} $Y$ is irreducible, and $Y$ is not in the closed matrix convex hull of
\[
 \cF_Y = \{X \in \cC: X \text{ is matrix extreme, } X \text{ is not a compression of } Y\}.
\]
\item\label{item:direct_limit} $Y$ is a matrix extreme point, and $Y$ is not a limit of compressions of $Z_k \in \cE_Y$.
\end{enumerate}
\end{theorem}

\begin{proof}
${\bf (\ref{item:blah}) \!\Rightarrow\! (\ref{item:compression_sent})}$: Every crucial matrix extreme point is irreducible, and $\cF_Y \subseteq \cE_Y$ holds, so the implication is trivial.

${\bf (\ref{item:compression_sent}) \!\Rightarrow\! (\ref{item:direct_limit})}$: Let $P_0$ be any matrix extreme point of $\cC$ which is not in the closed matrix hull of $\cF_Y$, noting such a point exists by  \cite[Theorem 4.3]{WW99}. By definition of $\cF_Y$, $P_0$ is a compression of $Y$, as is any matrix extreme dilation $P_0 \prec P_1 \prec P_2 \prec \ldots$ given by \cite[Lemma 6.12]{Kriel}. Since $Y$ is finite-dimensional, a contradiction is reached unless the dilation terminates in finitely many steps, i.e. some $P_k$ is absolute extreme. An absolute extreme point $P_k$ is a compression of $Y$, and such points admit only trivial dilations. However, $Y$ is assumed irreducible, so $Y \cong P_k$ is itself absolute extreme (and hence matrix extreme).

The remainder of the implication follows easily in the contrapositive, as if $m \geq \text{dim}(Y)$, $\cF_Y \cap M_{m}(\bC)^d = \cE_Y \cap M_m(\bC)^d$.

${\bf (\ref{item:direct_limit}) \!\Rightarrow\!  (\ref{item:blah})}$: This follows immediately from \cite[Theorem 4.6]{WW99}. If $Y$ is a matrix extreme point in the closed matrix convex hull of $\cE_Y$, then it is a limit of compressions of $Z_k \in \cE_Y$.
\end{proof}

%%%%%%%%%%%%%%%%%%%%%%%%%%%%

\section*{Acknowledgments}
We would like to thank the referee for helpful comments.

%%%%%%%%%%%%%%%%%%%%%%%%%%%%%%%%%%%%%%%
\bibliographystyle{amsplain}

\end{document}